\documentclass[a4paper,11pt,leqno]{smfart}

\usepackage{amssymb,amsmath,enumerate,verbatim,mathrsfs,graphics,graphicx,mathptm,float,mathtools}
\usepackage[T1]{fontenc}
\usepackage[latin1]{inputenc}
\usepackage[english, francais]{babel}
\usepackage[all]{xy}
\usepackage[pdfstartview=FitH,
            colorlinks=true,
            linkcolor=blue,
            urlcolor=blue,
            citecolor=red,
            bookmarks,
            bookmarksopen=true,
            bookmarksnumbered=true]{hyperref}

\usepackage{cleveref}
\theoremstyle{plain}

\newtheorem{theoalph}{Theorem}

\newtheorem{thmalph}[theoalph]{Theorem}

\newtheorem{coralph}[theoalph]{Corollary}

\theoremstyle{definition}

\theoremstyle{remark}

\theoremstyle{plain}
\newtheorem{thmsec}{Theorem}[section]
\newtheorem{thm}[thmsec]{Theorem}
\newtheorem{pro}[thmsec]{Proposition}
\newtheorem{lem}[thmsec]{Lemma}

\newtheorem{conj}{Conjecture}
\newtheorem{ques}{Question}

\theoremstyle{definition}
\newtheorem{defin}[thmsec]{Definition}

\theoremstyle{remark}
\newtheorem{rem}[thmsec]{Remark}
\newtheorem{rems}[thmsec]{Remarks}
\newtheorem{eg}[thmsec]{Example}

\def\og{\leavevmode\raise.3ex\hbox{$\scriptscriptstyle\langle\!\langle$~}}
\def\fg{\leavevmode\raise.3ex\hbox{~$\!\scriptscriptstyle\,\rangle\!\rangle$}}

\addtocounter{section}{0}             
\numberwithin{equation}{section}       

\newcommand{\vide}{\emptyset}
\newcommand{\N}{\mathbb{N}}
\newcommand{\Z}{\mathbb{Z}}

\newcommand{\C}{\mathbb{C}}

\newcommand{\pp}{\mathbb{P}^{2}_{\mathbb{C}}}
\newcommand{\pd}{\mathbb{\check{P}}^{2}_{\mathbb{C}}}

\newcommand\X{\mathrm{X}}

\newcommand\Sing{\mathrm{Sing}}
\newcommand\Tang{\mathrm{Tang}}

\newcommand\IF{\mathrm{I}_{\mathcal{F}}}
\newcommand\IinvF{\mathrm{I}_{\mathcal{F}}^{\mathrm{inv}}}
\newcommand\ItrF{\mathrm{I}_{\mathcal{F}}^{\hspace{0.2mm}\mathrm{tr}}}

\newcommand\F{\mathcal{F}}

\newcommand\omegaoverline{{\mspace{2mu}\overline{\mspace{-1.4mu}\omega\mspace{-1.4mu}}\mspace{2mu}}}

\newcommand\thetaoverline{{\mspace{2mu}\overline{\mspace{-1.4mu}\theta\mspace{-1.4mu}}\mspace{2mu}}}

\newcommand\px{\frac{\partial}{\partial x}}
\newcommand\py{\frac{\partial}{\partial y}}
\newcommand\pz{\frac{\partial}{\partial z}}

\begin{document}
\selectlanguage{english}
\title[Geometry of certain foliations on the complex projective plane]{Geometry of certain foliations on the complex projective plane}

\date{\today}

\author{Samir \textsc{Bedrouni}}

\address{Facult\'e de Math\'ematiques, USTHB, BP $32$, El-Alia, $16111$ Bab-Ezzouar, Alger, Alg\'erie}
\email{sbedrouni@usthb.dz,\,sbedrouni@mat.uab.es}

\author{David \textsc{Mar\'{\i}n}}

\thanks{This work has been partially funded by the Ministry of Science, Innovation and Universities of Spain through the grant PGC2018-095998-B-I00 and by the Agency for Management of University and Research Grants of Catalonia through the grant 2017SGR1725.}

\address{Departament de Matem\`{a}tiques, Edifici Cc, Universitat Aut\`{o}noma de Barcelona, 08193 Cerdanyola del Vall\`{e}s (Barcelona), Spain. Centre de Recerca Matem\`{a}tica, Edifici Cc, Campus de Bellaterra, 08193 Cerdanyola del Vall\`{e}s (Barcelona), Spain}

\email{davidmp@mat.uab.es}

\maketitle{}

\begin{abstract}
Let $d\geq2$ be an integer. The set $\mathbf{F}(d)$ of foliations of degree $d$ on the complex projective plane can be identified with a \textsc{Zariski}'s open set of a projective space of dimension $d^2+4d+2$ on which $\mathrm{Aut}(\mathbb P^2_{\mathbb C})$ acts. We~show that there are exactly two orbits $\mathcal{O}(\mathcal{F}_{1}^{d})$ and $\mathcal{O}(\mathcal{F}_{2}^{d})$ of minimal dimension $6$, necessarily closed in $\mathbf{F}(d)$. This generalizes known results in degrees $2$ and $3.$ We deduce that an orbit $\mathcal{O}(\mathcal{F})$ of an element $\mathcal{F}\in\mathbf{F}(d)$ of dimension $7$ is closed in $\mathbf{F}(d)$ if and only if $\mathcal{F}_{i}^{d}\not\in\overline{\mathcal{O}(\mathcal{F})}$ for $i=1,2.$ This allows us to show that in any degree $d\geq3$ there are closed orbits in $\mathbf F(d)$ other than the orbits $\mathcal{O}(\mathcal{F}_{1}^{d})$ and $\mathcal{O}(\mathcal{F}_{2}^{d}),$ unlike the situation in degree $2.$ On~the other hand, we introduce the notion of the basin of attraction $\mathbf{B}(\mathcal{F})$ of a foliation $\mathcal{F}\in\mathbf{F}(d)$ as the set of $\mathcal{G}\in\mathbf{F}(d)$ such that $\mathcal{F}\in\overline{\mathcal{O}(\mathcal{G})}.$ We show that the basin of attraction $\mathbf{B}(\mathcal{F}_{1}^{d})$, resp. $\mathbf{B}(\mathcal{F}_{2}^{d})$, contains a quasi-projective subvariety of $\mathbf{F}(d)$ of dimension greater than or equal to $\dim\mathbf{F}(d)-(d-1)$, resp.~$\dim \mathbf{F}(d)-(d-3)$. In particular, we obtain that the basin $\mathbf{B}(\F_{2}^{3})$ contains a non-empty \textsc{Zariski} open subset of $\mathbf{F}(3)$. This is an analog in degree $3$ of a result on foliations of degree~$2$ due to \textsc{Cerveau}, \textsc{D\'eserti}, \textsc{Garba Belko} and \textsc{Meziani}.

\end{abstract}

\section*{Introduction}

\noindent The set  $\mathbf{F}(d)$ of holomorphic foliations of degree $d$ on  $\pp$ is identified with a \textsc{Zariski} open subset of the projective space  $\mathbb{P}_{\C}^{\hspace{0.2mm}d^2+4d+2}$. We are interested here in the action of the group  $\mathrm{Aut}(\pp)=\mathrm{PGL}_3(\mathbb{C})$ on $\mathbf{F}(d).$  We generalize to arbitrary degree some results known in small degrees~\cite{CDGBM10,AR16,BM19} on this action.

\noindent For~$\F\in\mathbf{F}(d)$, we will respectively denote by  $\mathcal{O}(\F)$ and $\mathrm{Iso}(\F)$  the orbit and the isotropy group of $\F$ under the action of~$\mathrm{Aut}(\pp),$~{\it i.e.}
\begin{align*}
\mathcal{O}(\F):=\{\varphi^*\F\in\mathbf{F}(d)\hspace{1mm}\vert\hspace{1mm}\varphi\in\mathrm{Aut}(\pp)\}
&& \text{and} &&
\mathrm{Iso}(\F):=\{\varphi\in\mathrm{Aut}(\pp)\hspace{1mm}\vert\hspace{1mm}\varphi^*\F=\F\}.
\end{align*}
$\mathcal{O}(\F)$ is a \textsc{Zariski} irreducible subset of  $\mathbf{F}(d)$ and $\mathrm{Iso}(\F)$ is an algebraic subgroup of  $\mathrm{Aut}(\pp).$

\noindent Following   \cite{MP13} we will say that a foliation of $\mathbf{F}(d)$ is convex if its leaves other than straight
lines have no inflection points. We will denote by $\mathbf{FC}(d)$ the subset of $\mathbf{F}(d)$ consisting of convex foliations, which is a \textsc{Zariski} closed subset of~$\mathbf{F}(d).$

\noindent According to  \cite[Proposition~2.2]{Bru15} every foliation of degree~$0$ or~$1$ is convex. For $d\geq2$, $\mathbf{FC}(d)$ is a proper~closed subset of $\mathbf{F}(d)$ and it contains the foliation  $\F_{1}^{d}$ defined in the affine chart $(x,y)$ by the  $1$-form (see \cite[page~75]{Bed17}) $$\omega_{1}^{d}=y^{d}\mathrm{d}x+x^{d}(x\mathrm{d}y-y\mathrm{d}x).$$

\noindent We know from~\cite[Proposition~2.3]{CDGBM10} that if  $\F$ is an element of  $\mathbf{F}(d)$ with $d\geq2,$ then the dimension of~$\mathcal{O}(\F)$ is at least $6,$ or equivalently, the dimension of~$\mathrm{Iso}(\F)$ is at most $2$. In addition these bounds are attained by the convex foliation $\F_{1}^{d}$ and the non convex foliation $\F_{2}^{d}$ defined by the $1$-form (see \cite{Bed17}) $$\omega_{2}^{d}=x^{d}\mathrm{d}x+y^{d}(x\mathrm{d}y-y\mathrm{d}x).$$

\noindent The main result of this paper is the following.
\begin{thmalph}\label{thmalph:dim6}
{\sl Let $d$ be an integer greater than or equal to $2$ and let $\F$ be an element of  $\mathbf{F}(d).$ Assume that
the isotropy group $\mathrm{Iso}(\F)$ of $\F$ has dimension $2$. Then~$\F$ is linearly conjugated to one of the two foliations
 $\F_{1}^{d}$ and $\F_{2}^{d}$ defined respectively by the $1$-forms
\begin{itemize}
\item[\texttt{1. }] $\omega_{1}^{d}=y^{d}\mathrm{d}x+x^{d}(x\mathrm{d}y-y\mathrm{d}x);$

\item[\texttt{2. }] $\omega_{2}^{d}=x^{d}\mathrm{d}x+y^{d}(x\mathrm{d}y-y\mathrm{d}x).$
\end{itemize}
In other words, $\mathcal{O}(\F_{1}^{d})$ and $\mathcal{O}(\F_{2}^{d})$ are the only orbits of dimension $6$.
They are closed in~$\mathbf{F}(d).$ Moreover we have
\begin{align*}
&
\mathrm{Iso}(\F_{1}^{d})=\left\{
\left(\frac{\alpha^{d-1}x}{1+\beta x},\frac{\alpha^{d}y}{1+\beta x}\right)
\hspace{1mm}\Big\vert\hspace{1mm}
\alpha\in\mathbb{C}^*,\hspace{1mm}\beta\in\mathbb{C}\right\},\\
&
\mathrm{Iso}(\F_{2}^{d})=\left\{
\left(\frac{\alpha^{d+1}x}{1+\beta x},\frac{\alpha^{d}y}{1+\beta x}\right)
\hspace{1mm}\Big\vert\hspace{1mm}
\alpha\in\mathbb{C}^*,\hspace{1mm}\beta\in\mathbb{C}
\right\};
\end{align*}
these two groups are not conjugated.}
\end{thmalph}

\noindent This theorem is a generalization in arbitrary degree of previous results on foliations of degrees $d=2$~(\cite[Proposition~2.7]{CDGBM10}) and $d=3$ (\cite[Theorem~10]{AR16}, \cite[Corollary~B]{BM19}).

\noindent We  also obtain the following corollary, which generalizes~\cite[Corollary~3.9]{BM19}:
\begin{coralph}\label{coralph:dim-o(F)=<7}
{\sl Let $d$ be an integer greater than or equal to $2$ and let $\F$ be an element of  $\mathbf{F}(d).$ If~$\dim\mathcal{O}(\F)\leq~\hspace{-0.8mm}7$,~then
$$\overline{\mathcal{O}(\F)}\subset\mathcal{O}(\F)\cup\mathcal{O}(\F_{1}^{d})\cup\mathcal{O}(\F_{2}^{d}).$$
In particular, when  $\dim\mathcal{O}(\F)=7,$ the orbit $\mathcal{O}(\F)$ of $\F$ is closed in  $\mathbf{F}(d)$
if and only if $\F_{i}^{d}\not\in\overline{\mathcal{O}(\F)}$ for $i=1,2.$
}
\end{coralph}

\noindent In the spirit of Corollary~\ref{coralph:dim-o(F)=<7} we can ask under what condition the closure in $\mathbf{F}(d)$ of the orbit $\mathcal{O}(\F)$ of an element  $\F$ of $\mathbf{F}(d)$ contains the foliations $\F_{1}^{d}$ and $\F_{2}^{d}$, a question that we have already asked and studied in degree $3$ in~\cite[Section~3]{BM19}. In Section~\S\ref{sec:feuilletages-dgnr-F1-F2}, we extend (Propositions~\ref{pro:cond-nécess-suff-dgnr-F1}~and~\ref{pro:cond-nécess-suff-dgnr-F2})  in arbitrary degree $d$ our previous results in \cite[Propositions~3.10,~3.12,~3.15,~3.17]{BM19} concerning this question.

\noindent For $\F\in\mathbf{F}(d),$  we call {\sl basin of attraction} of $\F$ the subset  $\mathbf{B}(\F)$ of $\mathbf{F}(d)$ defined by
\begin{align*}
\mathbf{B}(\F):=\{\mathcal{G}\in \mathbf{F}(d)\hspace{1mm}\vert\hspace{1mm} \F\in\overline{\mathcal{O}(\mathcal{G})}\}.
\end{align*}

\noindent It follows from  \cite[Theorem~2.15]{CDGBM10} that in degree $2$ the basin $\mathbf{B}(\F_{1}^{2})$ contains a quasi-projective subvariety of $\mathbf{F}(2)$ of dimension greater than or equal to $\dim\mathbf{F}(2)-1.$ In Section~\S\ref{sec:feuilletages-dgnr-F1-F2}, we establish an analogous result in any degree greater than $2.$
\begin{thmalph}[\rm{Theorem~\ref{thm:bassin-attraction-F1}}]\label{thmalph:bassin-attraction-F1}
{\sl For any integer $d\geq2$, the basin of attraction $\mathbf{B}(\F_{1}^{d})$ of $\F_{1}^{d}$ contains a quasi-projective subvariety of $\mathbf{F}(d)$ of dimension greater than or equal to $\dim\mathbf{F}(d)-(d-1).$
}
\end{thmalph}

\noindent Notice that the non-convexity of  $\F_{2}^{d}$ and the fact that $\mathbf{FC}(d)$ is closed in $\mathbf{F}(d)$ imply that
\begin{align}\label{equa:Delta(d)<F(d)-FC(d)}
&\mathbf{B}(\F_{2}^{d})\subset\mathbf{F}(d)\setminus\mathbf{FC}(d).
\end{align}

\noindent In  degree $2$, according to \cite[Theorem~3]{CDGBM10}, inclusion~(\ref{equa:Delta(d)<F(d)-FC(d)}) is an equality:
\begin{align}\label{equa:Delta(2)=F(2)-FC(2)}
&\mathbf{B}(\F_{2}^{2})=\mathbf{F}(2)\setminus\mathbf{FC}(2).
\end{align}

\noindent It follows in particular from equality~(\ref{equa:Delta(2)=F(2)-FC(2)}) that the basin $\mathbf{B}(\F_{2}^{2})$ is a \textsc{Zariski} open subset of $\mathbf{F}(2).$ For $d\geq3$ we show the following result.
\begin{thmalph}[\rm{Theorem~\ref{thm:bassin-attraction-F2}}]\label{thmalph:bassin-attraction-F2}
{\sl In any degree  $d\geq3$, the basin of attraction $\mathbf{B}(\F_{2}^{d})$ of $\F_{2}^{d}$ contains a quasi-projective subvariety of $\mathbf{F}(d)$ of dimension greater than or equal to $\dim \mathbf{F}(d)-(d-3)$. In~particular, the basin $\mathbf{B}(\F_{2}^{3})$ contains a non-empty \textsc{Zariski} open subset of $\mathbf{F}(3).$}
\end{thmalph}

\noindent Along the same order of ideas, we prove the following result.
\begin{thmalph}[\rm{Theorem~\ref{thm:intersection-bassin-attraction-F1-F2}}]\label{thmalph:intersection-bassin-attraction-F1-F2}
{\sl For any integer $d\geq2$, the intersection $\mathbf{B}(\F_{1}^{d})\cap\mathbf{B}(\F_{2}^{d})$ is non-empty and it contains a quasi-projective subvariety of $\mathbf{F}(d)$ of dimension equal to $\dim\mathbf{F}(d)-3d.$
}
\end{thmalph}

\noindent By combining equality~(\ref{equa:Delta(2)=F(2)-FC(2)}) with the classification of C.~\textsc{Favre} and J.~V.~\textsc{Pereira} of convex foliations of degree two (\emph{cf.} \cite[Proposition~7.4]{FP15} or \cite[Theorem~A]{BM20}), we see that the only closed orbits in $\mathbf{F}(2)$ under the action of $\mathrm{Aut}(\pp)$ are those of $\F_{1}^{2}$ and $\F_{2}^{2}.$ We show in Section~\S\ref{sec:F0-lambda-degre-d} that in any degree $d\geq 3$ there are closed orbits in $\mathbf{F}(d)$ other than the orbits $\mathcal{O}(\F_{1}^{d})$ and $\mathcal{O}(\F_{2}^{d}),$ unlike the situation in degree $2.$ More precisely, we will consider a family of elements of $\mathbf{F}(d)$ which has been already studied in degree $d=2$ in \cite[page~189]{CDGBM10}, namely the family $(\F_{0}^{d}(\lambda))_{\lambda\in\C^*}$ of foliations of degree $d$ on $\pp$ defined by the $1$-form $$\omega_{0}^{d}(\lambda)=x\mathrm{d}y-\lambda y\mathrm{d}x+y^{d}\mathrm{d}y.$$ We will see that, for $\lambda=1,$  $\F_{0}^{d}(1)$ is linearly conjugated to the foliation  $\F_{1}^{d}$ and that, for any $\lambda\neq 1,$ $\dim\mathcal{O}\big(\F_{0}^{d}(\lambda)\big)=7.$ Moreover, we will show (Proposition~\ref{pro:Orbite-F0-lambda-degre-3-4-5-fermee}) that the orbit $\mathcal{O}\big(\mathcal{F}_{0}^{d}(\lambda)\big)$ is closed for any $d\geq3$ and $\lambda=-\frac{1}{d-1}$, resp.
for any $d\in\{3,4,5\}$ and any $\lambda\in\mathbb{C}^*,$ and we conjecture that it is so for any $d\geq6$ and any $\lambda\in\mathbb{C}^*$ (see Conjectures~\ref{conj:1} and~\ref{conj:2}).

\section{Some definitions and notations}
\bigskip

\subsection{Singularities and local invariants}

A degree $d$ holomorphic foliation $\F$ on $\pp$ is defined in homogeneous coordinates $[x:y:z]$ by a $1$-form $$\omega=a(x,y,z)\mathrm{d}x+b(x,y,z)\mathrm{d}y+c(x,y,z)\mathrm{d}z,$$ where $a,$ $b$ and $c$ are homogeneous polynomials of degree $d+1$ without common factor and satisfying the \textsc{Euler} condition $i_{\mathrm{R}}\omega=0$, where $\mathrm{R}=x\frac{\partial{}}{\partial{x}}+y\frac{\partial{}}{\partial{y}}+z\frac{\partial{}}{\partial{z}}$ denotes the radial vector field and $i_{\mathrm{R}}$ is the interior product by  $\mathrm{R}$.

\noindent Dually the foliation $\mathcal{F}$ can also be defined by a homogeneous vector field  $$\hspace{5mm}\mathrm{Z}=U(x,y,z)\frac{\partial{}}{\partial{x}}+V(x,y,z)\frac{\partial{}}{\partial{y}}+W(x,y,z)\frac{\partial{}}{\partial{z}},$$
the coefficients $U,V$ and $W$ are homogeneous polynomials of degree $d$ without common factor. The relation between $\mathrm{Z}$ and $\omega$ is given by
$$\hspace{-2.9cm}\omega=i_{\mathrm{R}}i_{\mathrm{Z}}(\mathrm{d}x\wedge\mathrm{d}y\wedge\mathrm{d}z).$$

\noindent The {\sl singular locus} $\mathrm{Sing}\mathcal{F}$ of $\mathcal{F}$ is the projectivization of the singular locus of~$\omega$ $$\mathrm{Sing}\,\omega=\{(x,y,z)\in\mathbb{C}^3\,\vert \, a(x,y,z)=b(x,y,z)=c(x,y,z)=0\}.$$

\noindent Let $\mathcal{C}\subset\pp$ be an algebraic curve with homogeneous equation $F(x,y,z)=0.$ We say that $\mathcal{C}$ is an \textsl{invariant curve} by $\F$ if $\mathcal{C}\smallsetminus\Sing\F$ is a union of (ordinary) leaves of the regular foliation $\F|_{\pp\smallsetminus\Sing\F}$. In algebraic terms, this is equivalent to require that the $2$-form  $\omega\wedge\mathrm{d}F$ is divisible by $F$, {\it i.e.} it vanishes along each irreducible component of $\mathcal{C}.$

\noindent Let $p$ be an arbitrary  point of $\mathcal{C}.$ When each irreducible component of $\mathcal{C}$ passing through $p$ is not $\F$-invariant, we define the \textsl{tangency order} $\Tang(\F,\mathcal{C},p)$ of $\F$ with $\mathcal{C}$ at $p$ as follows. We fix a local chart $(\mathrm{u},\mathrm{v})$ such that $p=(0,0)$; let $f(\mathrm{u},\mathrm{v})=0$ be a reduced local equation of  $\mathcal{C}$ in a neighborhood of $p$  and let $\X$ be a vector field defining the germ of  $\F$ at $p.$ We denote by $\X(f)$ the \textsc{Lie} derivative of $f$ along $\X$ and by $\langle f,\X(f)\rangle$ the ideal of  $\C\{\mathrm{u},\mathrm{v}\}$ generated by $f$ and $\X(f)$. Then $$\Tang(\F,\mathcal{C},p)=\dim_\mathbb{C}\frac{\C\{\mathrm{u},\mathrm{v}\}}{\langle f,\X(f)\rangle}.$$ Notice that $\Tang(\F,\mathcal{C},p)$ coincides with the intersection multiplicity $(\mathcal{C}.\mathcal{C}')_p$ at $p$ of the two algebraic curves $\mathcal{C}=\{F=0\}$ and $\mathcal{C}'=\{\mathrm{Z}(F)=0\}.$ Moreover, $\Tang(\F,\mathcal{C},p)<+\infty$ by the non-invariance of the irreducible components of~$\mathcal{C}$ passing through $p.$ By~convention, we put $\Tang(\F,\mathcal{C},p)=+\infty$ if there is at least one irreducible component of $\mathcal{C}$ invariant by $\F$ and passing through $p.$
\smallskip

\noindent Let us recall some local notions attached to the pair $(\mathcal{F},s)$, where  $s\in\Sing\mathcal{F}$. The germ of $\F$ at $s$ is defined, up to multiplication by a unity in the local ring  $\mathcal{O}_s$ at $s$, by a vector field~$\mathrm{X}=A(\mathrm{u},\mathrm{v})\frac{\partial{}}{\partial{\mathrm{u}}}+B(\mathrm{u},\mathrm{v})\frac{\partial{}}{\partial{\mathrm{v}}}$. The~{\sl algebraic multiplicity} $\nu(\mathcal{F},s)$ of $\mathcal{F}$ at $s$ is given by $$\nu(\mathcal{F},s)=\min\{\nu(A,s),\nu(B,s)\},$$ where $\nu(g,s)$ denotes the algebraic multiplicity of the function $g$ at $s.$ Let us denote by $\mathfrak{L}_s(\F)$ the family of straight lines through $s$ which are not invariant by $\F.$ For any line $\ell$ of~$\mathfrak{L}_s(\F),$ we have the inequalities $1\leq\Tang(\F,\ell_s,s)\leq d$. This allows us to associate to the pair~$(\mathcal{F},s)$ the following (invariant) integers
\begin{align*}
&\hspace{0.8cm}\tau(\mathcal{F},s)=\min\{\Tang(\F,\ell,s)\hspace{1mm}\vert\hspace{1mm}\ell\in\mathfrak{L}_s(\F)\},
&&\hspace{0.5cm}
\kappa(\mathcal{F},s)=\max\{\Tang(\F,\ell,s)\hspace{1mm}\vert\hspace{1mm}\ell\in\mathfrak{L}_s(\F)\}.
\end{align*}
The invariant $\tau(\mathcal{F},s)$ represents the tangency order of $\mathcal{F}$ with a generic line passing through $s.$ It is easy to see that
$$\tau(\mathcal{F},s)=\min\{k\geq 1\hspace{1mm}\vert\hspace{1mm}\det(J^{k}_{s}\,\mathrm{X},\mathrm{R}_{s})\not\equiv0\}\geq\nu(\mathcal{F},s),$$ where $J^{k}_{s}\,\mathrm{X}$ denotes the $k$-jet of $\mathrm{X}$ at $s$ and $\mathrm{R}_{s}$ is the radial vector field centered at $s.$ The \textsl{\textsc{Milnor} number} of $\F$ at $s$ is the integer $$\mu(\mathcal{F},s)=\dim_\mathbb{C}\mathcal{O}_s/\langle A,B\rangle,$$ where $\langle A,B\rangle$ denotes the ideal of $\mathcal{O}_s$ generated by $A$ and $B$.
\smallskip

\noindent The singularity $s$ is called {\sl radial of order} $n-1$ if $\nu(\mathcal{F},s)=1$ and $\tau(\mathcal{F},s)=n.$
\smallskip

\noindent The singularity $s$ is called {\sl non-degenerate} if $\mu(\F,s)=1,$ or equivalently if the Jacobian matrix of $\mathrm{X}$ at $s$, denoted by $\mathrm{Jac}\hspace{0.1mm}\mathrm{X}(s),$ possesses two nonzero eigenvalues $\lambda,\mu.$ In this case, the quantity
\begin{align*}
\mathrm{BB}(\F,s)=\frac{\mathrm{tr}^{2}(\mathrm{Jac}\hspace{0.1mm}\mathrm{X}(s))}{\det(\mathrm{Jac}\hspace{0.1mm}\mathrm{X}(s))}
=\frac{\lambda}{\mu}+\frac{\mu}{\lambda}+2
\end{align*}
is called the {\sl \textsc{Baum-Bott} index} of $\F$ at $s,$ see \cite{BB72}.
\smallskip

\noindent We will say that the singularity $s$ is {\sl quasi-radial of order} $n-1$ if $\mu(\F,s)=1,$ $\mathrm{BB}(\mathcal{F},s)=4$ and $\kappa(\mathcal{F},s)=n.$ In the sequel we will denote by
 $\mathrm{QRad}(\F,n-1)$ the set of quasi-radial singularities of  $\F$ of order $n-1$ and by  $\widehat{\mathrm{QRad}}(\F,n-1)$ the subset of  $\Sing(\F)\times\mathfrak{L}_s(\F)$ defined by
\begin{align*}
\widehat{\mathrm{QRad}}(\F,n-1):=\Big\{(s,\ell)\in\Sing(\F)\times\mathfrak{L}_s(\F)\hspace{1mm}\vert\hspace{1mm} \mu(\F,s)=1,\hspace{1mm}\mathrm{BB}(\F,s)=4,\hspace{1mm} \Tang(\F,\ell,s)=n\Big\}.
\end{align*}

\begin{rem}
Every radial singularity of order $n-1$ of a foliation $\F$ of degree $d\geq2$ on $\pp$ is quasi-radial of order $n-1.$ The converse is false: for instance, for the foliation defined in the affine chart $z=1$ by the $1$-form  $(x+y)\mathrm{d}y-y\mathrm{d}x+(x^n+y^d)\mathrm{d}x,$ with $n\in\{2,3,\ldots,d\},$ the point $[0:0:1]$ is a quasi-radial singularity of order $n-1,$ but it is not radial.
\end{rem}

\subsection{Inflection points}

Let us consider a foliation $\F$ of degree $d$ on $\pp$ and let $p$ be a regular~point~of~$\F$. Let us denote by  $\mathrm{T}^{\mathbb{P}}_{p}\F$ the tangent line to the leaf of $\F$ passing through $p$; it is the straight line of $\pp$ passing through $p$ with direction $\mathrm{T}_{\hspace{-0.4mm}p}\F.$
If $k\in\{2,\ldots,d\},$ we will say that $p$ is a {\sl (transverse) inflection point of order}~$k-1$ of $\F$ if $\Tang(\F,\mathrm{T}^{\mathbb{P}}_{p}\F,p)=k,$ in which case the line  $\mathrm{T}^{\mathbb{P}}_{p}\F$ is not invariant by $\F.$ When $\mathrm{T}^{\mathbb{P}}_{p}\F$ is $\F$-invariant, the point $p$ will be called  a {\sl trivial inflection point} of $\F$. If we denote by $\mathrm{Inv}(\F)$ the set of invariant lines
of $\F,$ then the set of trivial inflection points of  $\F$ is precisely $\mathrm{Inv}(\F)\setminus\Sing(\F).$
In the sequel, we will denote by  $\mathrm{Flex}(\F)$ the set of inflection points of  $\F$ and by $\mathrm{Flex}(\F,k-1)$ the subset of
 $\mathrm{Flex}(\F)$ consisting of transverse inflection points of $\F$ of order $k-1,$ {\it i.e.}
\begin{align*}
\mathrm{Flex}(\F,k-1):=\Big\{p\in\pp\hspace{1mm}\vert\hspace{1mm} p\not\in\Sing(\F),\hspace{1mm} \Tang(\F,\mathrm{T}^{\mathbb{P}}_{p}\F,p)=k\Big\}.
\end{align*}

\noindent Let us recall the notion of inflection divisor of $\F$, introduced by \textsc{Pereira}~\cite{Per01}, which allows to determine the set~$\mathrm{Flex}(\F).$ Let $\mathrm{Z}$ be a homogeneous vector field of degree $d$ on~$\mathbb{C}^3$ defining $\mathcal{F}.$ The~{\sl inflection divisor}~of~$\mathcal{F}$, denoted by $\IF$, is the divisor of $\pp$ defined by the homogeneous equation
\begin{equation}\label{equa:ext1}
\left| \begin{array}{ccc}
x   &  \mathrm{Z}(x)   &  \mathrm{Z}^2(x) \\
y   &  \mathrm{Z}(y)   &  \mathrm{Z}^2(y)  \\
z   &  \mathrm{Z}(z)   &  \mathrm{Z}^2(z)
\end{array} \right|=0.
\end{equation}
According to \cite{Per01}, $\IF$ satisfies the following properties:
\begin{enumerate}
\item [\texttt{1.}]
The support of $\IF$ is exactly the closure of the set  $\mathrm{Flex}(\F)$ of inflection points of $\F.$ More precisely, $\IF$ can be decomposed as  $\IF=\IinvF+\ItrF,$ where the support of $\IinvF$ is the set $\mathrm{Inv}(\F)$ of $\F$-invariant lines and the support of $\ItrF$ is the closure of the set of transverse inflection points of~$\mathcal{F}.$

\item [\texttt{2.}] If $\mathcal{C}$ is an algebraic curve invariant by $\mathcal{F},$ then $\mathcal{C}\subset\IF$ if and only if $\mathcal{C}\subset\mathrm{Inv}(\F).$

\item [\texttt{3.}] The degree of the divisor  $\IF$ is $3d.$
\end{enumerate}

\noindent The foliation $\mathcal{F}$ will be called {\sl convex} if its inflection divisor $\IF$ is totally invariant by $\mathcal{F}$, {\it i.e.} if $\IF$ is a product of invariant lines.

\section{Description of the foliations $\F$ of degree greater than or equal to $2$ such that $\dim\mathcal{O}(\F)=6$}

Recall that the foliations $\F_{1}^{d}$ and $\F_{2}^{d}$ are respectively defined in the affine chart $z=1$ by the $1$-forms
\begin{align*}
&\omega_{1}^{d}=y^{d}\mathrm{d}x+x^{d}(x\mathrm{d}y-y\mathrm{d}x)
&&\text{and}&&
\omega_{2}^{d}=x^{d}\mathrm{d}x+y^{d}(x\mathrm{d}y-y\mathrm{d}x).
\end{align*}
The foliation $\F_{1}^{d}$ is convex  with inflection divisor $\mathrm{I}_{\F_{1}^{d}}=\mathrm{I}_{\F_{1}^{d}}^{\hspace{0.2mm}\mathrm{inv}}=x^{d+1}y^{2d-1}$ and it has two singular points $s_1=[0:0:1]$ and $s_2=[0:1:0]$; the singularity $s_1$ has maximal algebraic multiplicity $d$ and $s_2$ is radial of maximal order $d-1.$ The foliation $\F_{2}^{d}$ is not convex with invariant inflection divisor $\mathrm{I}_{\F_{2}^{d}}^{\hspace{0.2mm}\mathrm{inv}}=x^{2d+1}$ and transverse inflection divisor $\mathrm{I}_{\F_{2}^{d}}^{\hspace{0.2mm}\mathrm{tr}}=y^{d-1}.$ The singular locus  $\Sing(\F_{2}^{d})$ is reduced to the point \mbox{$s_1=[0:0:1]$;} moreover $\nu(\F_{2}^{d},s_1)=d.$ We note that the $1$-forms $\dfrac{\omega_{1}^{d}}{x^2y^d}$ and $\dfrac{\omega_{2}^{d}}{x^{d+2}}$ are closed and they respectively admit as first integrals
\begin{align*}
&\frac{1}{d-1}\left(\frac{x}{y}\right)^{d-1}+\frac{1}{x}
&&\text{and}&&
\frac{1}{d+1}\left(\frac{y}{x}\right)^{d+1}-\frac{1}{x};
\end{align*}
this allows to check that
\begin{small}
\begin{align*}
&
\mathrm{Iso}(\F_{1}^{d})=\left\{
\left(\frac{\alpha^{d-1}x}{1+\beta x},\frac{\alpha^{d}y}{1+\beta x}\right)
\hspace{1mm}\Big\vert\hspace{1mm}
\alpha\in\mathbb{C}^*,\hspace{1mm}\beta\in\mathbb{C}\right\}
&&\text{\normalsize and}&&
\mathrm{Iso}(\F_{2}^{d})=\left\{
\left(\frac{\alpha^{d+1}x}{1+\beta x},\frac{\alpha^{d}y}{1+\beta x}\right)
\hspace{1mm}\Big\vert\hspace{1mm}
\alpha\in\mathbb{C}^*,\hspace{1mm}\beta\in\mathbb{C}
\right\}.
\end{align*}
\end{small}
In particular, $\dim\mathrm{Iso}(\F_{i}^{d})=2$ for $i=1,2.$ Thus the orbits $\mathcal{O}(\F_{1}^{d})$ and $\mathcal{O}(\F_{2}^{d})$ are both of  dimension~$6,$ which is the minimal dimension possible in any degree $d\geq 2$  (\cite[Proposition~2.3]{CDGBM10}). Theorem~\ref{thmalph:dim6} announced in  the Introduction shows that the orbits  $\mathcal{O}(\F_{1}^{d})$ and $\mathcal{O}(\F_{2}^{d})$ are the only orbits having minimal dimension~$6.$  The goal of this section is to prove this theorem.
\smallskip

\noindent  Let us denote by  $\chi(\pp)$ the \textsc{Lie} algebra of holomorphic vector fields on $\pp$: $\chi(\pp)$ is of course the
\textsc{Lie} algebra of the automorphism group of  $\pp.$ Let $\F$ be a foliation on  $\pp$ and let  $\mathrm{X}$  be an element of $\chi(\pp).$ Following~\cite{CDGBM10} we will say that $\mathrm{X}$ is a {\sl symmetry} of the foliation  $\mathcal{F}$ if the flow $\exp(t\mathrm{X})$ is, for each $t,$ in the isotropy group $\mathrm{Iso}(\mathcal{F})$ of~$\mathcal{F}.$ If $\omega$ defines $\mathcal{F}$ in an affine chart, $\mathrm{X}$ is a symmetry of~$\mathcal{F}$ if and only if $\mathrm{L}_\mathrm{X}\omega\wedge\omega=0,$ where $\mathrm{L}_\mathrm{X}\omega$ denotes the \textsc{Lie} derivative of $\omega$ along $\mathrm{X}.$
\begin{lem}\label{lem:symetrie-affine}
{\sl Let $\F$ be a foliation of degree $d$ on  $\pp$ and let $\mathrm{X}$ be a symmetry of $\F.$ Assume that there is an affine chart $\C^2\subset\pp$ such that the vector field $\mathrm{X}$ is affine ({\it i.e.} $\deg X\leq 1$) and let $\omega$ be a  $1$-form defining~$\F$ in this chart. Then there is a constant $\lambda\in\C$ such that $\mathrm{L}_{\mathrm{X}}\omega=\lambda\omega.$
}
\end{lem}

\begin{proof}
We will use an argument similar to one in~\cite[Proposition~2.5]{CDGBM10}. Since $\mathrm{L}_\mathrm{X}\omega\wedge\omega=0$ and~$\omega$ has isolated singularities, the  \textsc{De Rham}-\textsc{Saito} division theorem (\emph{cf.} \cite{Sai76} or \cite[Proposition~1.14]{CCD13}) ensures the existence of a holomorphic function $g$ on $\C^2$ such that  $\mathrm{L}_{\mathrm{X}}\omega=g\omega.$ The $1$-form $\omega$ and the vector field $\mathrm{X}$ being polynomials, $\mathrm{L}_{\mathrm{X}}\omega$ is also polynomial; therefore $g$ is rational and holomorphic on $\C^2$ hence polynomial. The~vector field $\mathrm{X}$ being affine we have $\deg\mathrm{L}_{\mathrm{X}}\omega\leq\deg\omega$; the equality $\mathrm{L}_{\mathrm{X}}\omega=g\omega$ implies that $g$ is constant.
\end{proof}
\medskip

\noindent If $\mathcal{F}$ is a foliation on  $\pp$, we will denote by $\mathfrak{iso}(\F)$ the \textsc{Lie} algebra of the algebraic group $\mathrm{Iso}(\F)$; $\mathfrak{iso}(\F)$ is a \textsc{Lie} subalgebra of~$\chi(\pp)$ and it consists of symmetries of~$\F.$ We know from
 \cite[Proposition 2.5]{CDGBM10} that if $\dim\mathfrak{iso}(\F)=2$ then $\mathfrak{iso}(\F)$ is {\sl affine}, {\it i.e.} generated by two vector fields $\mathrm{X}$ and $\mathrm{Y}$ such that $[\mathrm{X},\mathrm{Y}]=\mathrm{Y}.$ The~following lemma classifies the affine  \textsc{Lie} subalgebras of $\chi(\pp)$ and it will be used to prove Theorem~\ref{thmalph:dim6}.
\begin{lem}\label{lem:sous-algebres-affines}
{\sl Every affine \textsc{Lie} subalgebra of $\chi(\pp)$ is linearly conjugated to one of the following models
\begin{itemize}
\item[($\mathfrak{a}$)] $\big\langle \gamma\,x\px+y\py,\py \big\rangle$ with $\gamma\in\C^*;$

\item[($\mathfrak{b}$)] $\big\langle y\py,\py \big\rangle;$

\item[($\mathfrak{c}$)] $\big\langle \px+y\py,\py\big\rangle;$

\item[($\mathfrak{d}$)] $\big\langle x\px+(x+y)\py,\py\big\rangle;$

\item[($\mathfrak{e}$)] $\big\langle x\px+2y\py,\px+x\py \big\rangle.$
\end{itemize}
}
\end{lem}

\begin{proof}
Let $\mathfrak{g}$ be an affine  \textsc{Lie} subalgebra of $\chi(\pp).$ Then there exist $\mathrm{X}$ and $\mathrm{Y}$ in $\chi(\pp)$ such that $\mathfrak{g}=\langle\mathrm{X},\mathrm{Y}\rangle$ and $[\mathrm{X},\mathrm{Y}]=\mathrm{Y}.$ Fixing homogeneous coordinates  $[x:y:z]$ in $\pp$ we have an isomorphism of  \textsc{Lie}~algebras $\tau\hspace{0.1cm}\colon\mathfrak{sl}_{3}(\C)\rightarrow\chi(\pp)$ defined, for  $A\in\mathfrak{sl}_{3}(\C),$ by
\[
\tau(A)=(x\hspace{2mm} y\hspace{2mm} z)A
\left(\begin{array}{c}
\vspace{0.5mm}
\px\\
\vspace{0.5mm}
\py\\
\vspace{0.5mm}
\pz
\end{array}\right).
\]
Notice that if
$
A=\left(
\begin{array}{ccc}
a_{11} & a_{12} & a_{13}\\
a_{21} & a_{22} & a_{23}\\
a_{31} & a_{32} & a_{33}
\end{array}\right)\in\mathfrak{sl}_{3}(\C),
$
then in the affine chart $z=1$ the vector field $\tau(A)\in\chi(\pp)$ writes as 
\begin{align*}
\left(a_{31}+(a_{11}-a_{33})x+a_{21}y-a_{13}x^2-a_{23}xy\right)\px+\left(a_{32}+a_{12}x+(a_{22}-a_{33})y-a_{13}xy-a_{23}y^2\right)\py.
\end{align*}

\noindent Let  $M$ and $N$ be the  matrices of $\mathfrak{sl}_{3}(\C)$ associated to the vector fields $\mathrm{X}$ and $\mathrm{Y}$ respectively, {\it i.e.} $M=\tau^{-1}(\mathrm{X})$ and $N=\tau^{-1}(\mathrm{Y})$. Then the \textsc{Lie} bracket $[\mathrm{X},\mathrm{Y}]$ corresponds to $[M,N]:=MN-NM$ and therefore $[M,N]=N.$

\noindent Let us write~$M=\left(\begin{array}{ccc}
-m_{22}-m_{33} & m_{12} & m_{13}\\
m_{21}         & m_{22} & m_{23}\\
m_{31}         & m_{32} & m_{33}
\end{array}\right).$
Taking into account the possible \textsc{Jordan} forms of a matrix of $\mathfrak{sl}_{3}(\C),$ it suffices us to treat the following possibilities
\begin{align*}
N=\left(\begin{array}{ccc}
-a-b & 0 & 0\\
0 & a & 0\\
0 & 0 & b
\end{array}\right),
&&
N=\left(\begin{array}{ccc}
-2c & 0 & 0\\
0   & c & 0\\
0   & 1 & c
\end{array}\right),
&&
N=\left(\begin{array}{ccc}
0 & 0 & 0\\
0 & 0 & 0\\
0 & 1 & 0
\end{array}\right),
&&
N=\left(\begin{array}{ccc}
0 & 0 & 0\\
1 & 0 & 0\\
0 & 1 & 0
\end{array}\right).
\end{align*}
where $a,b\in\C,c\in\C^*,$ with $(a,b)\neq(0,0).$
\smallskip

\textbf{\textit{1.}} If
$N=\left(\begin{array}{ccc}
-a-b & 0 & 0\\
0 & a & 0\\
0 & 0 & b
\end{array}\right)$
then the equality $[M,N]=N$ implies that $a=b=0$: contradiction.
\smallskip

\textbf{\textit{2.}} If
$N=\left(\begin{array}{ccc}
-2c & 0 & 0\\
0   & c & 0\\
0   & 1 & c
\end{array}\right)$
then the  $(1,1)$ coefficient of the matrix $[M,N]-N$ is equal to $2c$ and is therefore nonzero: contradiction.
\smallskip

\textbf{\textit{3.}} Assume that
$N=\left(\begin{array}{ccc}
0 & 0 & 0\\
0 & 0 & 0\\
0 & 1 & 0
\end{array}\right)$; the equality $[M,N]=N$ then leads to
\begin{small}
$M=\left(\begin{array}{ccc}
1-2m_{33} & m_{12}   & 0\\
0         & m_{33}-1 & 0\\
m_{31}    & m_{32}   & m_{33}
\end{array}\right)$.
\end{small}
Up to replacing $M$ by $M-m_{32}N$ we can assume that $m_{32}=0.$ Now we will distinguish several eventualities:
\smallskip

\textbf{\textit{3.1.}} When $(3m_{33}-1)(3m_{33}-2)\neq0$ the matrix
$P=\left(\begin{array}{ccc}
3m_{33}-1 & m_{12}        & 0\\
0         & 3m_{33}-2     & 0\\
-m_{31}   & -m_{31}m_{12} & 3m_{33}-2
\end{array}\right)$
commutes with $N$ and
$P^{-1}MP=\left(\begin{array}{ccc}
1-2m_{33} & 0        & 0\\
0         & m_{33}-1 & 0\\
0         & 0        & m_{33}
\end{array}\right)$.
Thus $\mathfrak{g}$ is linearly conjugated to
\begin{align*}
&
\big\langle \tau(P^{-1}MP),\tau(N) \big\rangle=
\big\langle (1-3m_{33})x\tfrac{\partial}{\partial x}-y\tfrac{\partial}{\partial y},\tfrac{\partial}{\partial y} \big\rangle=
\big\langle \gamma\,x\tfrac{\partial}{\partial x}+y\tfrac{\partial}{\partial y},\tfrac{\partial}{\partial y} \big\rangle,
\qquad\text{where }
\gamma=3m_{33}-1\in\C^*.
\end{align*}
\smallskip

\textbf{\textit{3.2.}} Assume that $m_{33}=\frac{1}{3}.$ If $\delta\in\C^*$ then the matrix
$P=\left(\begin{array}{ccc}
\frac{1}{\delta} & -m_{12}       & 0\\
0           &  1            & 0\\
0           &  m_{12}m_{31} & 1
\end{array}\right)$
commutes with $N$ and
$P^{-1}MP=\left(\begin{array}{ccc}
\frac{1}{3}      &  0           & 0\\
0                & -\frac{2}{3} & 0\\
\frac{m_{31}}{\delta} &  0      & \frac{1}{3}
\end{array}\right)$.
As a result $\mathfrak{g}$ is linearly conjugated to
\begin{align*}
&
\big\langle \tau(P^{-1}MP),\tau(N) \big\rangle=
\big\langle \tfrac{m_{31}}{\delta}\tfrac{\partial}{\partial x}-y\tfrac{\partial}{\partial y},\tfrac{\partial}{\partial y} \big\rangle=
\big\langle -\tfrac{m_{31}}{\delta}\tfrac{\partial}{\partial x}+y\tfrac{\partial}{\partial y},\tfrac{\partial}{\partial y} \big\rangle.
\end{align*}
The case where $m_{31}=0$ leads to the model $(\mathfrak{b}).$ If  $m_{31}\neq0$ then by taking  $\delta=-m_{31}$ we get the model~$(\mathfrak{c}).$
\smallskip

\textbf{\textit{3.3.}} Assume that $m_{33}=\frac{2}{3}.$ If $\delta\in\C^*$ then the matrix
$P=\left(\begin{array}{ccc}
\delta         &  0            & 0\\
0              &  1            & 0\\
-\delta m_{31} & -m_{12}m_{31} & 1
\end{array}\right)$
commutes with $N$ and
$P^{-1}MP=\left(\begin{array}{ccc}
-\frac{1}{3} & \frac{m_{12}}{\delta} & 0\\
0            &  -\frac{1}{3}         & 0\\
0            &  0                    & \frac{2}{3}
\end{array}\right)$.
As a consequence $\mathfrak{g}$ is linearly conjugated to
\begin{align*}
&
\big\langle \tau(P^{-1}MP),\tau(N) \big\rangle=
\big\langle -x\tfrac{\partial}{\partial x}+(\tfrac{m_{12}}{\delta}x-y)\tfrac{\partial}{\partial y},\tfrac{\partial}{\partial y} \big\rangle=
\big\langle  x\tfrac{\partial}{\partial x}+(y-\tfrac{m_{12}}{\delta}x)\tfrac{\partial}{\partial y},\tfrac{\partial}{\partial y} \big\rangle.
\end{align*}
The case $m_{12}=0$ leads to the model $(\mathfrak{a})$ with $\gamma=1.$ If $m_{12}\neq0$ then by taking $\delta=-m_{12}$ we obtain the model~$(\mathfrak{d}).$
\smallskip

\textbf{\textit{4.}} Assume that
$N=\left(\begin{array}{ccc}
0 & 0 & 0\\
1 & 0 & 0\\
0 & 1 & 0
\end{array}\right)$;
then the equality $[M,N]=N$ implies that
$M=\left(\begin{array}{ccc}
-1     & 0      & 0\\
m_{32} & 0      & 0\\
m_{31} & m_{32} & 1
\end{array}\right)$.
Up to replacing $M$ by $M-m_{32}N$ we can assume that $m_{32}=0.$  The matrix
$P=\left(\begin{array}{ccc}
1                 & 0 & 0\\
0                 & 1 & 0\\
-\frac{m_{31}}{2} & 0 & 1
\end{array}\right)$
commutes with $N$ and
$P^{-1}MP=\left(\begin{array}{ccc}
-1 & 0 & 0\\
0  & 0 & 0\\
0  & 0 & 1
\end{array}\right)$.
Therefore $\mathfrak{g}$ is linearly conjugated to
\begin{align*}
&
\big\langle \tau(P^{-1}MP),\tau(N) \big\rangle=
\big\langle -2x\tfrac{\partial}{\partial x}-y\tfrac{\partial}{\partial y},y\tfrac{\partial}{\partial x}+\tfrac{\partial}{\partial y} \big\rangle=
\big\langle  2x\tfrac{\partial}{\partial x}+y\tfrac{\partial}{\partial y},y\tfrac{\partial}{\partial x}+\tfrac{\partial}{\partial y} \big\rangle.
\end{align*}
By permuting the coordinates $x$ and $y$ we obtain the model $(\mathfrak{e}).$
\end{proof}

\begin{proof}[\sl Proof of Theorem~\ref{thmalph:dim6}]
Since $\dim\mathfrak{iso}(\F)=\dim\mathrm{Iso}(\F)=2,$ \cite[Proposition~2.5]{CDGBM10} implies that $\mathfrak{iso}(\F)$ is affine. Therefore, up to linear conjugation, $\mathfrak{iso}(\F)$ is one of the models  ($\mathfrak{a}$)--($\mathfrak{e}$) of Lemma~\ref{lem:sous-algebres-affines}.

\noindent Let $\omega$ be a  $1$-form defining $\F$ in the affine chart $z=1$
$$\hspace{1cm}\omega=A(x,y)\mathrm{d}x+B(x,y)\mathrm{d}y,\quad A,B\in\mathbb{C}[x,y],\hspace{2mm}\gcd(A,B)=1.$$
We will study the five possible models ($\mathfrak{a}$)--($\mathfrak{e}$) of the  \textsc{Lie} algebra $\mathfrak{iso}(\F)$ and show that $\omega$ is linearly conjugated to one of the two $1$-forms $\omega_{1}^{d}$ or $\omega_{2}^{d}.$
\smallskip

\textbf{\textit{1.}} Assume that $\mathfrak{iso}(\F)$ is of one of the types ($\mathfrak{a}$)--($\mathfrak{d}$), {\it i.e.} that $\mathfrak{iso}(\F)=\big\langle \mathrm{X},\mathrm{Y}\big\rangle$ where $\mathrm{X}\in\{\gamma\,x\px+y\py,\varepsilon\px+y\py,x\px+(x+y)\py\},\mathrm{Y}=\py$ with $\varepsilon\in\{0,1\}$ and $\gamma\in\C^*.$ By Lemma~\ref{lem:symetrie-affine} there exist $\lambda,\mu\in\C$ such that $\mathrm{L}_{\mathrm{X}}\omega=\lambda\omega$ and $\mathrm{L}_{\mathrm{Y}}\omega=\mu\omega.$ Since $\mathrm{L}_{\mathrm{Y}}\mathrm{d}x=\mathrm{d}\mathrm{L}_{\mathrm{Y}}x=0$ and $\mathrm{L}_{\mathrm{Y}}\mathrm{d}y=\mathrm{d}\mathrm{L}_{\mathrm{Y}}y=0,$ we have $\mathrm{L}_{\mathrm{Y}}\omega=\mathrm{Y}(A)\mathrm{d}x+\mathrm{Y}(B)\mathrm{d}y=\frac{\partial A}{\partial y}\mathrm{d}x+\frac{\partial B}{\partial y}\mathrm{d}y.$ Therefore $\mathrm{L}_{\mathrm{Y}}\omega=\mu\omega$ if and only if $\frac{\partial A}{\partial y}=\mu A$ and $\frac{\partial B}{\partial y}=\mu B.$ Since $A,B\in\C[x,y]$ and~$\mu\in\C,$ it follows that $\mu=0$, $A(x,y)=A(x)$ and $B(x,y)=B(x).$ Thus
$$\hspace{1cm}\omega=A(x)\mathrm{d}x+B(x)\mathrm{d}y,\quad A,B\in\mathbb{C}[x],\hspace{2mm}\gcd(A,B)=1.$$

\textbf{\textit{1.1.}} Let us consider the case where $\mathrm{X}=\gamma\,x\px+y\py$ with $\gamma\in\C^*.$ We have
\begin{align*}
\mathrm{L}_{\mathrm{X}}\omega=
\mathrm{X}(A)\mathrm{d}x+A\mathrm{d}\mathrm{X}(x)+\mathrm{X}(B)\mathrm{d}y+B\mathrm{d}\mathrm{X}(y)
=\big(\gamma\,xA'+\gamma\,A\big)\mathrm{d}x+\big(\gamma\,xB'+B\big)\mathrm{d}y,
\end{align*}
so that $\mathrm{L}_{\mathrm{X}}\omega=\lambda\omega$ if and only if  $\gamma\,xA'=(\lambda-\gamma)A$ and $\gamma\,xB'=(\lambda-1)B.$ By putting $\kappa=\frac{\lambda-\gamma}{\gamma}$ and~$\nu=\frac{\lambda-1}{\gamma},$ the last two equations can be rewritten as $xA'=\kappa\,A$ and $xB'=\nu B$ and can be immediately integrated to give $A(x)=\alpha\,x^{\kappa}$ and $B(x)=\beta\,x^{\nu},$ where $\alpha,\beta\in\C.$ Since $A,B\in\mathbb{C}[x]$ and $\gcd(A,B)=1,$ we deduce that $\alpha,\beta\in\C^*,$ $\kappa,\nu\in\N$ and~$\kappa\,\nu=0.$ The equality $\deg\F=d$ then implies that
\begin{itemize}
\item either $\kappa=0$ and $\nu=d,$ in which case $\omega=\alpha\mathrm{d}x+\beta\,x^{d}\mathrm{d}y;$

\item or $\nu=0$ and $\kappa=d,$ in which case $\omega=\alpha\,x^d\mathrm{d}x+\beta\mathrm{d}y.$
\end{itemize}

\noindent If $\omega=\alpha\mathrm{d}x+\beta\,x^{d}\mathrm{d}y,$ resp. $\omega=\alpha\,x^d\mathrm{d}x+\beta\mathrm{d}y,$ by making the change of coordinates $\left(x,y\right)\mapsto\left(\frac{y}{x},-\frac{\alpha}{\beta\,x}\right)$, we reduce ourselves to $\omega=\omega_{1}^{d}=y^{d}\mathrm{d}x+x^{d}(x\mathrm{d}y-y\mathrm{d}x),$ resp. $\omega=\omega_{2}^{d}=x^{d}\mathrm{d}x+y^{d}(x\mathrm{d}y-y\mathrm{d}x).$
\smallskip

\textbf{\textit{1.2.}} Let us examine the case where $X=\varepsilon\px+y\py$ with  $\varepsilon\in\{0,1\}.$ Since $\mathrm{L}_{\mathrm{X}}\mathrm{d}x=\mathrm{d}\mathrm{L}_{\mathrm{X}}x=0$ and $\mathrm{L}_{\mathrm{X}}\mathrm{d}y=\mathrm{d}\mathrm{L}_{\mathrm{X}}y=\mathrm{d}y,$ we have $\mathrm{L}_{\mathrm{X}}\omega=\mathrm{X}(A)\mathrm{d}x+\mathrm{X}(B)\mathrm{d}y+B\mathrm{d}y=\varepsilon\,A'\mathrm{d}x+(\varepsilon\hspace{0.2mm}B'+B)\mathrm{d}y.$ Therefore $\mathrm{L}_{\mathrm{X}}\omega=\lambda\omega$ if and~only~if~$\varepsilon\,A'=\lambda\,A$ and $\varepsilon\hspace{0.2mm}B'=(\lambda-1)B.$ Since $A,B\in\C[x]$ and $\lambda\in\C,$ it follows that $AB=0$: contradiction with $\gcd(A,B)=1.$
\smallskip

\textbf{\textit{1.3.}} Let us study the eventuality: $\mathrm{X}=x\px+(x+y)\py.$ We have $\mathrm{d}\mathrm{X}(x)=\mathrm{d}x$ and $\mathrm{d}\mathrm{X}(y)=\mathrm{d}x+\mathrm{d}y,$ so that
\[
\mathrm{L}_{\mathrm{X}}\omega=\mathrm{X}(A)\mathrm{d}x+A\mathrm{d}x+\mathrm{X}(B)\mathrm{d}y+B(\mathrm{d}x+\mathrm{d}y)=(xA'+A+B)\mathrm{d}x+(xB'+B)\mathrm{d}y.
 \]
The condition $\mathrm{L}_{\mathrm{X}}\omega=\lambda\omega$ is then equivalent to the system of differential equations $xA'+B=(\lambda-1)A$ and~$xB'=(\lambda-1)B,$ which can be easily integrated to yield $A(x)=(a-b\ln x)x^{\lambda-1}$ and $B(x)=bx^{\lambda-1},$ where~$a,b\in\C.$ Since $A\in\C[x]$, we deduce that $b=0$ and therefore $B\equiv0$: contradiction with $\gcd(A,B)=1.$
\smallskip

\textbf{\textit{2.}} Assume that $\mathfrak{iso}(\F)$ is of type ($\mathfrak{e}$), {\it i.e.} $\mathfrak{iso}(\F)=\big\langle \mathrm{X},\mathrm{Y}\big\rangle$ where $\mathrm{X}=x\px+2y\py,\mathrm{Y}=\px+x\py.$ As before by writing explicitly  that $\mathrm{L}_{\mathrm{X}}\omega=\lambda\omega$ and $\mathrm{L}_{\mathrm{Y}}\omega=\mu\omega,$ with $\lambda,\mu\in\C$ (Lemma~\ref{lem:symetrie-affine}), we obtain the system of partial differential equations
\begin{align*}
&
x\tfrac{\partial A}{\partial x}+2y\tfrac{\partial A}{\partial y}=(\lambda-1)A,&&
x\tfrac{\partial B}{\partial x}+2y\tfrac{\partial B}{\partial y}=(\lambda-2)B,&&
\tfrac{\partial A}{\partial x}+x\tfrac{\partial A}{\partial y}=\mu\hspace{0.4mm}A-B,&&
\tfrac{\partial B}{\partial x}+x\tfrac{\partial B}{\partial y}=\mu\hspace{0.2mm}B.
\end{align*}

\noindent It follows in particular that
\begin{align*}
&
(x^2-2y)\tfrac{\partial B}{\partial y}=(\mu\,x+2-\lambda)B
&& \text{and} &&
(x^2-2y)\tfrac{\partial A}{\partial y}=(\mu\,x+1-\lambda)A-xB.
\end{align*}

\noindent Elementary integrations then lead to
\begin{align*}
&
B(x,y)=b(x)(x^2-2y)^{\frac{\lambda-2-\mu\,x}{2}}
&& \text{and} &&
A(x,y)=\left(a(x)\sqrt{x^2-2y}-xb(x)\right)(x^2-2y)^{\frac{\lambda-2-\mu\,x}{2}},
\end{align*}
for some functions $a$ and $b$ of the coordinate $x.$ Since $A,B\in\C[x,y]$ and $\gcd(A,B)=1,$ we deduce that $\lambda-2-\mu\,x=0$ and $a(x)=0$ for any $x\in\C$, hence $\lambda=2,\mu=0$ and $a\equiv0.$ Therefore $B(x,y)=b(x)$ and $A(x,y)=-xb(x)=-xB(x,y)$: contradiction with $\gcd(A,B)=1.$
\end{proof}

\vspace{0.2cm}

\section{Foliations of $\mathbf{F}(d)$ degenerating onto $\F_{1}^{d}$ and $\F_{2}^{d}$}\label{sec:feuilletages-dgnr-F1-F2}
In this section we will study the problem of knowing whether the closure of the orbit of a foliation of $\mathbf{F}(d)$ contains the foliations $\F_1^d$ and $\F_2^d.$ The following definition will be useful.

\begin{defin}[\cite{CDGBM10}]
Let $\F$ and $\F'$ be two foliations of~$\mathbf{F}(d).$ We say that $\F$ {\sl degenerates} onto $\F'$ if the closure $\overline{\mathcal{O}(\F)}$ (in~$\mathbf{F}(d)$) of the orbit $\mathcal{O}(\F)$ contains~$\mathcal{O}(\F')$ and $\mathcal{O}(\mathcal{F})\not=\mathcal{O}(\mathcal{F}').$
\end{defin}

\begin{rems}\label{rems:degenerescence}
Let $\F$ and $\F'$ be two foliations such that  $\F$ degenerates onto $\F'$. Then:
\begin{itemize}
\item [(i)]   $\dim\mathcal{O}(\F')<\dim\mathcal{O}(\F)$;

\item [(ii)] $\deg\IinvF\leq\deg\mathrm{I}_{\mathcal{F}'}^{\mathrm{inv}},$ which is equivalent to~$\deg\ItrF\geq\deg\mathrm{I}_{\mathcal{F}'}^{\hspace{0.2mm}\mathrm{tr}}$. In particular, if $\F$ is convex, $\F'$ is also convex.
\end{itemize}
\end{rems}

\noindent Corollary~\ref{coralph:dim-o(F)=<7} is an immediate consequence of Theorem~\ref{thmalph:dim6} and Remark~\ref{rems:degenerescence}~(i).
\begin{rem}\label{rem:adherence-orbite-homogene}
Corollary~\ref{coralph:dim-o(F)=<7} applies particularly to any foliation of $\mathbf{F}(d)$ which is {\sl homogeneous}, {\it i.e.} which is given, for a suitable choice of affine coordinates $(x,y),$ by a homogeneous $1$-form
$\omega=A(x,y)\mathrm{d}x+B(x,y)\mathrm{d}y,$ where $A,B\in\mathbb{C}[x,y]_d$ and $\gcd(A,B)=1.$ Indeed, for such a foliation  $\mathcal{H}\in \mathbf{F}(d)$, we have (\emph{cf.}~\cite{BM18})
\[
\mathrm{Iso}(\mathcal{H})=\{(\alpha\,x,\alpha\,y)\hspace{1mm}\vert\hspace{1mm}\alpha\in\mathbb{C}^*\};
\]
in particular, $\dim\mathcal{O}(\mathcal{H})=7$ and consequently
\[
\hspace{5.5mm}\overline{\mathcal{O}(\mathcal{H})}\subset\mathcal{O}(\mathcal{H})\cup\mathcal{O}(\F_{1}^{d})\cup\mathcal{O}(\F_{2}^{d}).
\]
\end{rem}

\noindent Assertion~\textbf{\textit{1.}} (resp. \textbf{\textit{2.}}) of the following proposition gives a necessary (resp. sufficient) condition for a~foliation of~$\mathbf{F}(d)$ to degenerate onto the foliation~$\F_{1}^{d}.$
\begin{pro}\label{pro:cond-nécess-suff-dgnr-F1}
{\sl Let $\F$ be an element of $\mathbf{F}(d)$ such that $\F_{1}^{d}\not\in\mathcal{O}(\F).$ The following assertions hold:

\noindent\textbf{\textit{1.}}~If $\F$ degenerates onto $\F_{1}^{d}$, then $\F$ possesses a non-degenerate singularity $m$ satisfying $\mathrm{BB}(\F,m)=4.$

\noindent\textbf{\textit{2.}}~If $\F$ possesses a quasi-radial singularity of maximal order $d-1$, {\it i.e.} if $\mathrm{QRad}(\F,d-1)\neq\vide$, then $\F$ degenerates onto $\F_1^d.$
}
\end{pro}

\begin{proof}
\textbf{\textit{1.}} Assume that  $\F$ degenerates onto $\F_{1}^{d}$. Then there is an analytic family of foliations $(\F_\varepsilon)$ defined by a family of $1$-forms $(\omega_\varepsilon)$ such that $\F_\varepsilon$ belongs to $\mathcal{O}(\F)$ for $\varepsilon\neq 0$ and $\F_{\varepsilon=0}=\F_{1}^{d}.$ The non-degenerate singular point of $\F_{1}^{d}$, denoted by~$m_0,$ is \og stable\fg\hspace{0.2mm} in the sense that there is an analytic family $(m_\varepsilon)$ of non-degenerate singular points of  $\F_\varepsilon$ such that $m_{\varepsilon=0}=m_0.$ The $\F_\varepsilon$'s being conjugated to $\F$ for $\varepsilon\neq 0$, the foliation $\F$ admits a non-degenerate singular point $m$ such that
$$
\forall\hspace{1mm}\varepsilon\in\mathbb{C}^*,\hspace{1mm}\mathrm{BB}(\F_\varepsilon,m_\varepsilon)=\mathrm{BB}(\F,m).
$$
Since $\mu(\F_\varepsilon,m_\varepsilon)=1$ for any $\varepsilon\in\mathbb{C}$, the function $\varepsilon\mapsto\mathrm{BB}(\F_\varepsilon,m_\varepsilon)$ is continuous, hence constant on $\mathbb{C}$. As a~consequence
\begin{align*}
&&\mathrm{BB}(\F,m)=\mathrm{BB}(\F_{\varepsilon=0},m_{\varepsilon=0})=\mathrm{BB}(\F_{1}^{d},m_0)=4.
\end{align*}
\smallskip

\noindent\textbf{\textit{2.}} Assume that $\F$ has a quasi-radial singularity $m$ of order $d-1$. Then $\mu(\F,m)=1,\mathrm{BB}(\F,m)=4$ and~$\kappa(\F,m)=d.$ This last equality ensures the existence of a line $\ell_m$ passing through $m$, not invariant by $\F$ and such that $\Tang(\F,\ell_m,m)=d.$ Let us choose an affine coordinate system  $(x,y)$ such that $m=(0,0)$ and~$\ell_m=\{x=0\}.$ The foliation $\F$ is defined in these coordinates by a $1$-form $\omega$ of type
\begin{align*}
\hspace{0.5cm}&\omega=C_{d}(x,y)(x\mathrm{d}y-y\mathrm{d}x)+\sum_{i=1}^{d}\big(A_i(x,y)\mathrm{d}x+B_i(x,y)\mathrm{d}y\big),&&
\text{where }A_i,\,B_i\in\mathbb{C}[x,y]_{i},\hspace{1mm}C_{d}\in\mathbb{C}[x,y]_{d}.
\end{align*}
We have
\begin{align*}
\omega\wedge\mathrm{d}x\Big|_{x=0}=\sum_{i=1}^{d}B_i(0,y)\mathrm{d}y\wedge\mathrm{d}x=\sum_{i=1}^{d}B_i(0,1)y^{i}\mathrm{d}y\wedge\mathrm{d}x.
\end{align*}
Then the equality~$\Tang(\F,\ell_m,m)=d$ translates into $B_i(0,1)=0$ for $i\in\{1,2,\ldots,d-1\}$ and $B_d(0,1)\neq0.$ This allows to write
\begin{align*}
&B_1(x,y)=\alpha\hspace{0.2mm}x,&&
B_{d}(x,y)=x\widehat{B}_{d-1}(x,y)+\gamma\,y^d,&&
B_{i}(x,y)=x\widetilde{B}_{i-1}(x,y)\hspace{1mm}\text{for }i\in\{2,3,\ldots,d-1\},
\end{align*}
where $\widetilde{B}_{i-1}\in\mathbb{C}[x,y]_{i-1},\hspace{1mm}\widehat{B}_{d-1}\in\mathbb{C}[x,y]_{d-1},\hspace{1mm}\gamma\in\mathbb{C}^*,\hspace{1mm} \alpha\in\mathbb{C}.$ The equalities $\mu(\F,m)=1$ and $\mathrm{BB}(\F,m)=4$ imply that $\alpha\neq0$ and $A_1(x,y)=\delta\hspace{0.2mm}x-\alpha y$ for some $\delta\in\mathbb{C}.$ Thus $\omega$ is of type
\begin{align*}
\hspace{0.5cm}&\omega=\delta\hspace{0.2mm}x\mathrm{d}x+
\big(x\widehat{B}_{d-1}(x,y)+\gamma\,y^d\big)\mathrm{d}y+
\big(C_{d}(x,y)+\alpha\big)\big(x\mathrm{d}y-y\mathrm{d}x\big)+
\sum_{i=2}^{d}A_i(x,y)\mathrm{d}x+x\sum_{i=2}^{d-1}\widetilde{B}_{i-1}(x,y)\mathrm{d}y,
\end{align*}
where $A_i\in\mathbb{C}[x,y]_{i},\hspace{1mm}
\widetilde{B}_{i-1}\in\mathbb{C}[x,y]_{i-1},\hspace{1mm}
\widehat{B}_{d-1}\in\mathbb{C}[x,y]_{d-1},\hspace{1mm}
\delta\in\mathbb{C},\hspace{1mm}
\alpha,\gamma\in\mathbb{C}^*.$
\vspace{2mm}

\noindent By putting $\varphi=\left(\varepsilon^d\hspace{0.1mm}x,\varepsilon\hspace{0.1mm}y\right)$ and $\theta=\alpha(x\mathrm{d}y-y\mathrm{d}x)+\gamma\,y^d\mathrm{d}y$, we obtain
\begin{SMALL}
\begin{align*}
\frac{1}{\varepsilon^{d+1}}\varphi^*\omega=\theta+
\varepsilon^{d-1}\left(\delta\hspace{0.2mm}x\mathrm{d}x+x\widehat{B}_{d-1}(\varepsilon^{d-1}x,y)\mathrm{d}y\right)+
\varepsilon^{d}C_{d}(\varepsilon^{d-1}x,y)\left(x\mathrm{d}y-y\mathrm{d}x\right)+
\sum_{i=2}^{d}\varepsilon^{i-1}A_i(\varepsilon^{d-1}x,y)\mathrm{d}x+x\sum_{i=2}^{d-1}\varepsilon^{i-1}\widetilde{B}_{i-1}(\varepsilon^{d-1}x,y)\mathrm{d}y
\end{align*}
\end{SMALL}
\hspace{-1.17mm}which tends to $\theta$ as $\varepsilon$ tends to $0.$ By making the change of coordinates $\left(x,y\right)\mapsto\left(\frac{x}{y}-\frac{\gamma}{\alpha y},\frac{x}{y}\right)$, we reduce ourselves to~$\theta=\omega_{1}^{d}=y^{d}\mathrm{d}x+x^{d}(x\mathrm{d}y-y\mathrm{d}x)$. As a result $\F$ degenerates onto $\F_{1}^{d}.$
\end{proof}

\begin{eg}\label{eg:H1}
Let us consider the homogeneous foliation $\mathcal{H}_{1}^{d}$ defined in the affine chart $z=1$ by the $1$-form
$$\omegaoverline_{1}^{d}=y^d\mathrm{d}x-x^d\mathrm{d}y.$$
We know from~\cite[Proposition~4.1]{BM18} that $\mathcal{H}_{1}^{d}$ is convex and admits the points $[1:0:0]$ and $[0:1:0]$
as radial singularities of maximal order $d-1.$ Therefore $\mathcal{H}_{1}^{d}$ degenerates onto~$\F_{1}^{d}$ (Proposition~\ref{pro:cond-nécess-suff-dgnr-F1}) and it does not degenerate onto $\F_{2}^{d}$, because $\F_{2}^{d}$ is not convex. Thus, according to Remark~\ref{rem:adherence-orbite-homogene}, we have
\[
\hspace{1.5cm}\overline{\mathcal{O}(\mathcal{H}_{1}^{d})}=\mathcal{O}(\mathcal{H}_{1}^{d})\cup\mathcal{O}(\F_{1}^{d}).
\]
\end{eg}

\begin{eg}\label{eg:famille-G-gamma}
Let us consider the family $(\mathcal{G}^{d}(\gamma))_{\gamma\in\C}$ of foliations of degree  $d$ on $\pp$ defined in the affine chart  $z=1$ by
$$\eta^{d}(\gamma)=(x-\gamma\,y)\mathrm{d}y-y\mathrm{d}x+x^d\mathrm{d}x-y^d\mathrm{d}y.$$
We remark that the point $m=[0:0:1]$ is a non-degenerate singularity of  $\mathcal{G}^{d}(\gamma)$ with \textsc{Baum-Bott} index~$4.$ Moreover, along the line $\ell=\{y=0\}$ we have $\eta^{d}(\gamma)\wedge\mathrm{d}y\Big|_{y=0}=x^d\mathrm{d}x\wedge\mathrm{d}y,$ so that $\Tang(\mathcal{G}^{d}(\gamma),\ell,m)=d$. It follows that the singularity  $m$ of $\mathcal{G}^{d}(\gamma)$ is quasi-radial of maximal order $d-1.$ As a consequence $\mathcal{G}^{d}(\gamma)$ degenerates onto~$\F_1^d$ (Proposition~\ref{pro:cond-nécess-suff-dgnr-F1}).
\end{eg}

\noindent The converse of assertion~\textbf{\textit{2.}} of Proposition~\ref{pro:cond-nécess-suff-dgnr-F1} is false as the following example shows.
\begin{eg}\label{eg:sans-singularité-kappa=d}
Let $\F$ be the foliation of degree $d\geq2$ on $\pp$ defined in the  affine chart $z=1$ by
\begin{align*}
&\omega=x\mathrm{d}y-y\mathrm{d}x+P(y)\mathrm{d}y,
\end{align*}
where $P$ is a polynomial of $\C[y]$ of degree $d$ admitting  $0$ as a root of multiplicity $\leq d-1,$ {\it i.e.} $P$ is of the form
\begin{align*}
&P(y)=y^{\nu}(a_{0}+a_{1}y+\cdots+a_{d-\nu}y^{d-\nu}),&&
\text{where}\hspace{1mm}\nu\in\{1,2,\ldots,d-1\},\hspace{1mm}a_{i}\in\mathbb{C},\hspace{1mm}a_{0}a_{d-\nu}\neq0.
\end{align*}

\noindent The singular locus of  $\F$ consists of the two points  $m=[0:0:1]$ and $m'=[1:0:0]$; moreover
\begin{align*}
&\mu(\F,m)=1,&& \mathrm{BB}(\F,m)=4,&&\kappa(\F,m)=\nu<d,&& \mu(\F,m')>1.
\end{align*}
It follows that $\F$ has no quasi-radial singularity of maximal order $d-1,$ {\it i.e.}  $\mathrm{QRad}(\F,d-1)=\vide.$
However, $\F$ degenerates onto $\F_{1}^{d}$. Indeed, by putting
$\varphi=$\begin{small}$\left(\dfrac{a_{d-\nu}}{\varepsilon^d}x,\dfrac{1}{\varepsilon}y\right)$\end{small}, we see that
\begin{align*}
&\lim_{\varepsilon\to 0}\frac{\varepsilon^{d+1}}{a_{d-\nu}}\varphi^*\omega=x\mathrm{d}y-y\mathrm{d}x+y^{d}\mathrm{d}y.
\end{align*}
\end{eg}

\begin{ques}\label{ques:1}
{\sl Let $\F$ be a foliation of degree $d\geq 2$ on $\pp.$ Is it true that if $\F$ degenerates onto $\F_1^d$ then
\begin{itemize}
\item [--] either $\F$ admits a quasi-radial singularity of maximal order $d-1$,

\item [--] or $\F$ is conjugated to Example~\ref{eg:sans-singularité-kappa=d}, {\it i.e.} up to linear conjugation $\F$ is given by a $1$-form of type~\, $x\mathrm{d}y-y\mathrm{d}x+P(y)\mathrm{d}y$ with $P\in\C[y],$ $\deg P=d$ and $P(0)=0$?
\end{itemize}
}
\end{ques}

\begin{pro}\label{pro:ouvert-U1}
{\sl
Let $d$ be an integer greater than or equal to $2.$ Let us denote by $U_1(d)$ the subset of  $\mathbf{F}(d)$ defined by
\begin{align*}
&U_{1}(d):=\Big\{\F\in \mathbf{F}(d)\hspace{1mm}\vert\hspace{1mm}\forall\hspace{0.2mm}s\in\Sing(\F),\hspace{1mm}\mu(\F,s)=1,\tau(\F,s)=1\Big\}.
\end{align*}
Then:
\begin{itemize}
\item [\textit{(i)}] $U_1(d)$ is a non-empty \textsc{Zariski} open subset of~$\mathbf{F}(d);$ in particular, for any $\gamma\in\C,$ $\mathcal{G}^{d}(\gamma)\in U_1(d)$ if and only if $\gamma\left(\gamma^{d+1}+\dfrac{(d+1)^{d+1}}{d^d}\right)\neq0.$
\smallskip
\item [\textit{(ii)}] Let $\F$ be an element of  $U_1(d).$ For any singular point  $s\in\Sing(\F),$ the set~$$\Lambda(\F,s):=\Big\{\ell_s\in\mathfrak{L}_s(\F)\hspace{1mm}\vert\hspace{1mm} \Tang(\F,\ell_s,s)>1\Big\}$$ has at most $2$ elements. In particular, the set $\bigcup\limits_{n=2}^{d}\widehat{\mathrm{QRad}}(\F,n-1)$ is finite.
\end{itemize}
}
\end{pro}

\noindent To prove this proposition, we need the following lemma.
\begin{lem}\label{lem:Tang-F-ell-s}
{\sl
Let $\F$ be a foliation of degree $d\geq 2$ on $\pp,$ $s$ a singular point of $\F,$ $\ell_s$ a line passing through~$s$ and not invariant by $\F$ and  $\mathrm{X}=A(x,y)\frac{\partial{}}{\partial{x}}+B(x,y)\frac{\partial{}}{\partial{y}}$ a polynomial vector field defining $\F$ in an affine chart~$(x,y)$ containing $s.$ Let us denote by $(x_0,y_0)$ the coordinates of  $s$ and let $a(x-x_0)+b(y-y_0)=0$ be an equation of the line $\ell_s.$ Then, for any $n\in\{2,3,\ldots,d\},$ $\Tang(\F,\ell_s,s)\geq n$ if and only if
\begin{align*}
\frac{\mathrm{d}^j}{\mathrm{d}t^j}
\Big(a\hspace{0.2mm}A(x_0+b\hspace{0.2mm}t,y_0-a\hspace{0.2mm}t)+bB(x_0+b\hspace{0.2mm}t,y_0-a\hspace{0.2mm}t)\Big)\Big|_{t=0}=0,
\quad \forall\hspace{0.2mm}j\in\{1,2,\ldots,n-1\}.
\end{align*}
In particular, the set~$\Lambda(\F,s):=\Big\{\ell_s\in\mathfrak{L}_s(\F)\hspace{1mm}\vert\hspace{1mm} \Tang(\F,\ell_s,s)>\tau(\F,s)\Big\}$ is finite and its cardinality
is at~most~$\tau(\F,s)+1.$
}
\end{lem}

\begin{proof}
The  $1$-form $\omega=A(x,y)\mathrm{d}y-B(x,y)\mathrm{\mathrm{d}}x$ also defines the foliation $\F$ because $i_\mathrm{X}\omega=0.$ We have
\begin{align*}
\omega\wedge\mathrm{d}\big(a(x-x_0)+b(y-y_0)\big)\Big|_{(x,y)=(x_0+b\hspace{0.2mm}t,y_0-a\hspace{0.2mm}t)}=P(t)\mathrm{d}y\wedge\mathrm{d}x,
\end{align*}
where $P(t)=a\hspace{0.2mm}A(x_0+b\hspace{0.2mm}t,y_0-a\hspace{0.2mm}t)+bB(x_0+b\hspace{0.2mm}t,y_0-a\hspace{0.2mm}t)$. Thus $\Tang(\F,\ell_s,s)=\nu(P(t),0).$ Notice~that $P(0)=~\hspace{-0.5mm}0$ because the point $s$ being singular for $\F,$ we have $A(x_0,y_0)=B(x_0,y_0)=0.$ Then $\Tang(\F,\ell_s,s)\geq n$ if and only if the root $t=0$ of the polynomial $P$ has multiplicity at least $n$, that is if and only if $P'(0)=P''(0)=\cdots=P^{(n-1)}(0)=0$, hence the announced equivalence holds.
\medskip

\noindent By conjugating $\omega$ by the translation $(x+x_0,y+y_0)$, we can assume that $s=(0,0)$. Let us denote $\tau(\F,s)$ simply by $\tau$. Then the vector field $\X$ decomposes in the form
\begin{align*}
\X=C_{\tau-2}(x,y)\mathrm{R}+\sum_{i=\tau}^{d+1}\mathrm{X}_i,
\end{align*}
where $\mathrm{R}=x\frac{\partial{}}{\partial{x}}+y\frac{\partial{}}{\partial{y}}$, $C_{\tau-2}$ is a polynomial of degree~$\leq \tau-2,$ $\mathrm{X}_i=A_i(x,y)\frac{\partial{}}{\partial{x}}+B_i(x,y)\frac{\partial{}}{\partial{y}}$ is a homogeneous vector field of degree  $i$, with $\det(\X_{\tau},\mathrm{R})\not\equiv0.$ Thus, we have
\begin{small}
\begin{align*}
a\hspace{0.2mm}A(b\hspace{0.2mm}t,-a\hspace{0.2mm}t)+bB(b\hspace{0.2mm}t,-a\hspace{0.2mm}t)
&=\bigg(a\Big(xC_{\tau-2}(x,y)+\sum_{i=\tau}^{d+1}A_i(x,y)\Big)+b\Big(yC_{\tau-2}(x,y)+\sum_{i=\tau}^{d+1}B_i(x,y)\Big)\bigg)
\bigg|_{(x,y)=(b\hspace{0.2mm}t,-a\hspace{0.2mm}t)}\\
&=\sum_{i=\tau}^{d+1}\Big(a\hspace{0.2mm}A_i(b\hspace{0.2mm}t,-a\hspace{0.2mm}t)+bB_i(b\hspace{0.2mm}t,-a\hspace{0.2mm}t)\Big)\\
&=\sum_{i=\tau}^{d+1}t^iQ_{i+1}(a,b),
\end{align*}
\end{small}
\hspace{-1.2mm}where $Q_{i+1}(a,b):=a\hspace{0.2mm}A_i(b,-a)+bB_i(b,-a)$ is a homogeneous polynomial of degree $i+1$ in $(a,b).$ From~this, we deduce that $\Tang(\F,\ell_s,s)>\tau$ if and only if $Q_{\tau+1}(a,b)=0.$ As a result
\begin{align*}
\Lambda(\F,s)=\Big\{\ell_s=\{ax+by=0\}\in\mathfrak{L}_s(\F)\hspace{1mm}\vert\hspace{1mm} Q_{\tau+1}(a,b)=0\Big\}.
\end{align*}
Now, the polynomial  $Q_{\tau+1}$ is not identically zero because $Q_{\tau+1}(a,b)=-\det(\X_{\tau},\mathrm{R})\big|_{(x,y)=(b,-a)}\not\equiv0.$ It follows that $\Lambda(\F,s)$ has cardinality at most $\tau+1.$
\end{proof}

\begin{proof}[\sl Proof of Proposition~\ref{pro:ouvert-U1}] We have
\begin{align*}
&U_{1}(d)=\Big\{
\F\in \mathbf{F}(d)
\hspace{1mm}\vert\hspace{1mm}\forall\hspace{0.2mm}s\in\Sing(\F),
\hspace{1mm}
\det(\mathrm{Jac}\hspace{0.1mm}\mathrm{X}(s))\neq0,
\hspace{1mm}
\det(J^{1}_{s}\mathrm{X},\mathrm{R}_{s})\not\equiv0
\Big\},
\end{align*}
where $\mathrm{X}$ denotes a polynomial vector field defining $\F$ in an affine chart containing $s$ and $\mathrm{R}_s$ is the radial vector field centered at $s.$ It~follows that $U_{1}(d)$ is a \textsc{Zariski} open subset of $\mathbf{F}(d).$
To establish assertion~\textit{(i)}, it~remains to show that for any $\gamma\in\C,$ $\mathcal{G}^{d}(\gamma)\in U_1(d)$ if and only if $\gamma\left(\gamma^{d+1}+\frac{(d+1)^{d+1}}{d^d}\right)\neq0.$ In homogeneous coordinates, the foliation
 $\mathcal{G}^{d}(\gamma)$ is defined by the $1$-form
\begin{align*}
&\hspace{1cm}
\Omega^{d}(\gamma)=
z\big(x^d-yz^{d-1}\big)\mathrm{d}x-z\big(y^d+\gamma\,yz^{d-1}-xz^{d-1}\big)\mathrm{d}y+\big(y^{d+1}-x^{d+1}+\gamma\,y^2z^{d-1}\big)\mathrm{d}z\hspace{0.5mm}.
\end{align*}
The singular locus $\Sing\big(\mathcal{G}^{d}(\gamma)\big)$ consists of the points
\begin{align*}
&s_{0}=[0:0:1],&& s_{k}=[x_{k}:x_{k}^{d}:1],&& s'_{l}=[1:\xi^{l}:0],&& k\in\{1,2,\ldots,d^2-1\},\, l\in\{0,1,\ldots,d\},
\end{align*}
where $\xi=\exp\hspace{-0.4mm}\left(\frac{2\mathrm{i}\pi}{d+1}\right)$ and the $x_k$'s are the roots of the polynomial $P(x)=x^{d^2-1}+\gamma\,x^{d-1}-1.$

\noindent In the affine chart $z=1,$ resp. $x=1,$ $\mathcal{G}^{d}(\gamma)$ is given by the vector field
\begin{align*}
&\mathrm{Y}=\big(y^d+\gamma\,y-x\big)\frac{\partial{}}{\partial{x}}+\big(x^d-y\big)\frac{\partial{}}{\partial{y}},&&
\text{resp. } \mathrm{Z}=\big(y^{d+1}+\gamma\,y^2z^{d-1}-1\big)\frac{\partial{}}{\partial{y}}+z\big(y^d+\gamma\,yz^{d-1}-z^{d-1}\big)\frac{\partial{}}{\partial{z}}\hspace{0.5mm}.
\end{align*}
A direct computation show that
\begin{small}$\det(\mathrm{Jac}\hspace{0.1mm}\mathrm{Y}(s_0))=1\neq0$\end{small}, \begin{small}$\det(J^{1}_{s_0}\mathrm{Y},\mathrm{R}_{s_0})=\gamma\,y^2$\end{small} and
\begin{Small}
\begin{align*}
&
\det(\mathrm{Jac}\hspace{0.1mm}\mathrm{Z}(s'_l))=(d+1)\xi^{-2l}\neq0,&&
\det(\mathrm{Jac}\hspace{0.1mm}\mathrm{Y}(s_k))=1-d\gamma\,x_k^{d-1}-d^2x_k^{d^2-1}=(d-1)\big(d\gamma\,x_k^{d-1}-d-1\big),
\hspace{1mm}\text{\normalsize because}\hspace{1mm}P(x_k)=0,\\
&
\det(J^{1}_{s'_l}\mathrm{Z},\mathrm{R}_{s'_l})=d\xi^{-l}\big(y-\xi^{l}\big)z\not\equiv0,&&
\det(J^{1}_{s_k}\mathrm{Y},\mathrm{R}_{s_k})=\big(dx_k^{d^2-d}+\gamma\big)\big(y-x_k^d\big)^2-dx_k^{d-1}\big(x-x_k\big)^2\not\equiv0,
\hspace{1mm}\text{\normalsize because}\hspace{1mm}x_k\neq0.
\end{align*}
\end{Small}
\hspace{-1mm}From these we deduce that $\mathcal{G}^{d}(\gamma)\in U_1(d)$ if and only if  $\gamma\neq0$ and $d\gamma\,x_k^{d-1}-d-1\neq0$, {\it i.e.} if and only if $\gamma\neq0$ and $x_k^{d-1}\neq\frac{d+1}{d\gamma}.$
Now, by putting $Q(t)=t^{d+1}+\gamma\,t-1,$ we have $P(x)=Q(x^{d-1})$ so that $t_0\in\C$ is a root of the polynomial $Q(t)$ if and only if there exists $k\in\{1,2,\ldots,d^2-1\}$ such that $t_0=x_k^{d-1}.$ It follows that
\begin{align*}
\mathcal{G}^{d}(\gamma)\in U_1(d)\Longleftrightarrow
\gamma\,Q\left(\frac{d+1}{d\gamma}\right)\neq0\Longleftrightarrow
\gamma\left(\gamma^{d+1}+\frac{(d+1)^{d+1}}{d^d}\right)\neq0.
\end{align*}

\noindent Assertion~\textit{(ii)} is an immediate consequence of Lemma~\ref{lem:Tang-F-ell-s}.
\end{proof}

\begin{thm}\label{thm:bassin-attraction-F1}
{\sl Let $d$ be an integer greater than or equal to $2.$ Let us denote by $\Sigma_{1}(d)$ the subset of $\mathbf{F}(d)$ defined by
\begin{align*}
&\Sigma_{1}(d):=\Big\{\F\in \mathbf{F}(d)\hspace{1mm}\vert\hspace{1mm}\mathrm{QRad}(\F,d-1)\neq\vide\Big\}.
\end{align*}
Then
\begin{itemize}
\item [\textit{(a)}] $\emptyset\neq \Sigma_{1}(d)\varsubsetneq \mathbf{B}(\F_{1}^{d});$
\smallskip
\item [\textit{(b)}] $\Sigma_{1}(d)$ is a constructible subset of~$\mathbf{F}(d)$ of dimension greater than or equal to $\dim\mathbf{F}(d)-(d-1).$
\end{itemize}
}
\end{thm}

\begin{proof}
\textit{(a)} $\Sigma_{1}(d)$ contains the foliations $\mathcal{H}_{1}^{d}$ and $\mathcal{G}^{d}(\gamma),\gamma\in\C$ (Examples~\ref{eg:H1} and \ref{eg:famille-G-gamma}) and is therefore non-empty. Assertion~\textbf{\textit{2.}} of Proposition~\ref{pro:cond-nécess-suff-dgnr-F1} ensures that  $\Sigma_{1}(d)\subset\mathbf{B}(\F_{1}^{d})$; this  inclusion is strict as  Example~\ref{eg:sans-singularité-kappa=d} shows.
\medskip

\textit{(b)} Let us denote by $\pd$ the dual projective plane of $\pp.$ Let $\pi\hspace{1mm}\colon\mathbf{F}(d)\times\pp\times\pd\to\mathbf{F}(d)$ be the projection onto the first factor; we have $\Sigma_{1}(d)=\pi(W_{1}(d)),$ where
\begin{small}
\begin{align*}
W_{1}(d):&=\bigcup_{\F\in\Sigma_{1}(d)}\{\F\}\times\widehat{\mathrm{QRad}}(\F,d-1)\\
&=\Big\{
(\F,s,\ell)\in\mathbf{F}(d)\times\pp\times\pd\hspace{1mm}\vert\hspace{1mm}
s\in\Sing(\F),\hspace{1mm}\ell\in\mathfrak{L}_s(\F),\hspace{1mm}\mu(\F,s)=1,\hspace{1mm}\mathrm{BB}(\F,s)=4,\hspace{1mm} \Tang(\F,\ell,s)=d
\Big\}.
\end{align*}
\end{small}
\hspace{-1mm}According to Lemma~\ref{lem:Tang-F-ell-s}, $W_{1}(d)$ can be rewritten as
\begin{Small}
\begin{align}\label{equa:W1(d)}
W_{1}(d)=\left\{
(\F,s,\ell)\in\mathbf{F}(d)\times\pp\times\pd
\hspace{0.5mm}\left\vert
\begin{array}[c]{l}
\vspace{2mm}
s=(x_0,y_0)\in\ell=\{a(x-x_0)+b(y-y_0)=0\}\\
\vspace{2mm}
A(x_0,y_0)=0,
\hspace{1mm}
B(x_0,y_0)=0,
\hspace{1mm}
\det(\mathrm{Jac}\hspace{0.1mm}\mathrm{X}(s))\neq0,
\hspace{1mm}
\dfrac{\mathrm{tr}^{2}(\mathrm{Jac}\hspace{0.1mm}\mathrm{X}(s))}{\det(\mathrm{Jac}\hspace{0.1mm}\mathrm{X}(s))}=4\\
\vspace{2mm}
a\hspace{0.2mm}A(x_0+b\hspace{0.2mm}t,y_0-a\hspace{0.2mm}t)+bB(x_0+b\hspace{0.2mm}t,y_0-a\hspace{0.2mm}t)\not\equiv0\\
\dfrac{\mathrm{d}^j}{\mathrm{d}t^j}
\Big(a\hspace{0.2mm}A(x_0+b\hspace{0.2mm}t,y_0-a\hspace{0.2mm}t)+bB(x_0+b\hspace{0.2mm}t,y_0-a\hspace{0.2mm}t)\Big)\Big|_{t=0}=0,j=1,\ldots,d-1
\end{array}
\right.
\right\},
\end{align}
\end{Small}
\hspace{-1mm}where $\mathrm{X}=A(x,y)\frac{\partial{}}{\partial{x}}+B(x,y)\frac{\partial{}}{\partial{y}}$ is a polynomial vector field defining $\F$ in an affine chart $(x,y)$ containing $s.$ It~follows that $W_{1}(d)$ is a quasi-projective subvariety of  $\mathbf{F}(d)\times\pp\times\pd.$ Thus, by \textsc{Chevalley}'s~Theorem \cite[Exercise~II.3.19]{Har77}, the set $\Sigma_{1}(d)=\pi(W_{1}(d))$ is constructible.
\medskip

\noindent According to the above discussion and Proposition~\ref{pro:ouvert-U1}~\textit{(i)}, the intersection $U_1(d)\cap\Sigma_{1}(d)$ contains the foliations $\mathcal{G}^{d}(\gamma)$, with $\gamma\left(\gamma^{d+1}+\frac{(d+1)^{d+1}}{d^d}\right)\neq0,$ and is therefore non-empty ($U_1(d)$ being the set of $\F\in\mathbf{F}(d)$ such that for any $s\in\Sing\F,$ $\mu(\F,s)=1$ and $\tau(\F,s)=1$). Then there exists an irreducible component $\Sigma_{1}^{0}(d)$ of $\Sigma_{1}(d)$ such that $U_1(d)\cap\Sigma_{1}^{0}(d)\neq\vide.$ Let $W_{1}(d)=\bigcup\limits_{i=1}^{k}W_{1}^{i}(d)$ be the decomposition of $W_{1}(d)$ into its irreducible components. Let us denote by $\pi_0\hspace{1mm}\colon W_{1}(d)\to\mathbf{F}(d)$ the~restriction of $\pi$ to~$W_{1}(d).$ Then, there is $n\in\{1,\ldots,k\}$ such that $\overline{\pi_0(W_{1}^{n}(d))}=\overline{\Sigma_{1}^{0}(d)}.$ Indeed, since  $\Sigma_{1}(d)=\pi_0(W_{1}(d)),$ we have $\overline{\Sigma_{1}^{0}(d)}\subset\overline{\Sigma_{1}(d)}=\bigcup\limits_{i=1}^{k}\overline{\pi_0(W_{1}^{i}(d))}.$ The irreducibility of $\Sigma_{1}^{0}(d)$ therefore ensures the existence of~$n\in\{1,\ldots,k\}$ such that $\overline{\Sigma_{1}^{0}(d)}\subset\overline{\pi_0(W_{1}^{n}(d))}\subset\overline{\Sigma_{1}(d)}.$ Since $\overline{\Sigma_{1}^{0}(d)}$ is an irreducible component of~$\overline{\Sigma_{1}(d)}$ and since $\overline{\pi_0(W_{1}^{n}(d))}$ is irreducible by continuity of $\pi_0,$ we deduce that $\overline{\pi_0(W_{1}^{n}(d))}=\overline{\Sigma_{1}^{0}(d)}.$

\noindent Thus, since $U_1(d)$ is a \textsc{Zariski} open subset of $\mathbf{F}(d)$ (Proposition~\ref{pro:ouvert-U1}~\textit{(i)}), the morphism $\pi_0$ induces by restriction a dominant morphism of quasi-projective varieties $\pi_{0}^{n}\hspace{1mm}\colon W_{1}^{n}(d)\cap\pi_{0}^{-1}(U_1(d))\to\overline{\Sigma_{1}^{0}(d)}\cap U_1(d).$

\noindent Notice that all the fibers of $\pi_0$ over the elements of  $U_1(d)\cap\Sigma_{1}(d)$ are finite and non-empty. Indeed, if~$\F\in U_1(d)\cap\Sigma_{1}(d)$ then, by Proposition~\ref{pro:ouvert-U1}~\textit{(ii)}, the set~$\widehat{\mathrm{QRad}}(\F,d-1)$ is finite and non-empty; therefore so~is~$\pi_{0}^{-1}(\F)=\{\F\}\times\widehat{\mathrm{QRad}}(\F,d-1).$ Since $\pi_{0}(W_{1}^{n}(d)\cap\pi_{0}^{-1}(U_1(d)))\subset U_1(d)\cap\Sigma_{1}(d),$ it follows that all the non-empty fibers of $\pi_{0}^{n}$ are finite and therefore zero-dimensional. The fiber dimension theorem (\emph{cf.} \cite[Theorem~3, page 49]{Mum88}) then implies that $\dim (W_{1}^{n}(d)\cap\pi_{0}^{-1}(U_1(d)))=\dim (\overline{\Sigma_{1}^{0}(d)}\cap U_1(d));$ since $W_{1}^{n}(d)\cap\pi_{0}^{-1}(U_1(d))$ and $\overline{\Sigma_{1}^{0}(d)}\cap U_1(d)$ are non-empty open subsets of the irreducible varieties  $W_{1}^{n}(d)$ and~$\overline{\Sigma_{1}^{0}(d)}$~respectively, we have
\begin{align*}
\hspace{2.3cm}\dim\overline{\Sigma_{1}^{0}(d)}=\dim (\overline{\Sigma_{1}^{0}(d)}\cap U_1(d))=\dim (W_{1}^{n}(d)\cap\pi_{0}^{-1}(U_1(d)))=\dim W_{1}^{n}(d).
\end{align*}
Now, from~(\ref{equa:W1(d)}) we deduce that each irreducible component  $W_{1}^{i}(d)$ of $W_{1}(d)$ has dimension
\[\hspace{1.45cm}\dim W_{1}^{i}(d)\geq \dim(\mathbf{F}(d)\times\pp\times\pd)-4-(d-1)=\dim\mathbf{F}(d)-(d-1),\]
hence
\[\hspace{1.57cm}\dim\Sigma_{1}(d)=\dim\overline{\Sigma_{1}(d)}\geq\dim\overline{\Sigma_{1}^{0}(d)}=\dim W_{1}^{n}(d)\geq\dim\mathbf{F}(d)-(d-1).\]
\end{proof}

\noindent Assertion~\textbf{\textit{1.}} (resp. \textbf{\textit{2.}}) of the following proposition gives a necessary (resp. sufficient) condition for a~foliation of~$\mathbf{F}(d)$ to degenerate onto the foliation~$\F_{2}^{d}.$
\begin{pro}\label{pro:cond-nécess-suff-dgnr-F2}
{\sl Let $\F$ be an element of $\mathbf{F}(d)$ such that $\F_{2}^{d}\not\in\mathcal{O}(\F).$ The following assertions hold:

\noindent\textbf{\textit{1.}} If $\F$ degenerates onto $\F_{2}^{d}$, then $\deg\ItrF\geq d-1.$

\noindent\textbf{\textit{2.}} If $\F$ admits an inflection point of maximal order $d-1$, {\it i.e.} if $\mathrm{Flex}(\F,d-1)\neq\emptyset,$ then $\F$ degenerates onto~$\F_{2}^{d}.$
}
\end{pro}

\begin{proof}
\textbf{\textit{1.}} If $\F$ degenerates onto $\F_{2}^{d},$ then $\deg\ItrF\geq\deg\mathrm{I}_{\F_{2}^{d}}^{\hspace{0.2mm}\mathrm{tr}}.$ An immediate computation shows that $\mathrm{I}_{\F_{2}^{d}}^{\hspace{0.2mm}\mathrm{tr}}=~\hspace{-0.5mm}y^{d-1}$ so that $\deg\mathrm{I}_{\F_{2}^{d}}^{\hspace{0.2mm}\mathrm{tr}}=d-1$, hence the announced inequality holds.
\smallskip

\noindent\textbf{\textit{2.}} Assume that  $\F$ possesses such a point.
We choose an affine coordinate system $(x,y)$ such that $p=(0,0)$ is an inflection point of order  $d-1$ of $\F$ and $x=0$ is the tangent line to the leaf of  $\F$ passing through $p.$ Let~$\omega$~be~a~$1$-form defining  $\F$ in these coordinates. Since $\mathrm{T}^{\mathbb{P}}_{p}\F=\{x=0\}$, $\omega$ is of type
\begin{align*}
&\omega=C_{d}(x,y)(x\mathrm{d}y-y\mathrm{d}x)+\alpha\mathrm{d}x+\sum_{i=1}^{d}\big(A_i(x,y)\mathrm{d}x+B_i(x,y)\mathrm{d}y\big),&&
\text{where }A_i,\,B_i\in\mathbb{C}[x,y]_{i},\hspace{1mm}C_{d}\in\mathbb{C}[x,y]_{d},\hspace{1mm}\alpha\in\mathbb{C}^*.
\end{align*}
We have
\begin{align*}
\omega\wedge\mathrm{d}x\Big|_{x=0}=\sum_{i=1}^{d}B_i(0,y)\mathrm{d}y\wedge\mathrm{d}x=\sum_{i=1}^{d}B_i(0,1)y^{i}\mathrm{d}y\wedge\mathrm{d}x.
\end{align*}
Therefore the hypothesis that $(0,0)$ is an inflection point of order $d-1$ of $\F$ translates into $B_i(0,1)=0$ for~$i\in\{1,2,\ldots,d-1\}$ and $B_d(0,1)\neq0.$ Then we can write
\begin{align*}
&B_{d}(x,y)=x\widehat{B}_{d-1}(x,y)+\beta y^d,&&
B_{i}(x,y)=x\widetilde{B}_{i-1}(x,y)\hspace{1mm}\text{for }i\in\{1,2,\ldots,d-1\},
\end{align*}
where $\widetilde{B}_{i-1}\in\mathbb{C}[x,y]_{i-1},\hspace{1mm}
\widehat{B}_{d-1}\in\mathbb{C}[x,y]_{d-1},\hspace{1mm}
\beta\in\mathbb{C}^*.$
Thus $\omega$ is of type
\begin{align*}
\hspace{0.5cm}&\omega=\alpha\mathrm{d}x+
\big(x\widehat{B}_{d-1}(x,y)+\beta y^d\big)\mathrm{d}y+
C_{d}(x,y)\big(x\mathrm{d}y-y\mathrm{d}x\big)+
\sum_{i=1}^{d}A_i(x,y)\mathrm{d}x+x\sum_{i=1}^{d-1}\widetilde{B}_{i-1}(x,y)\mathrm{d}y,
\end{align*}
where $A_i\in\mathbb{C}[x,y]_{i},\hspace{1mm}
\widetilde{B}_{i-1}\in\mathbb{C}[x,y]_{i-1},\hspace{1mm}
\widehat{B}_{d-1}\in\mathbb{C}[x,y]_{d-1},\hspace{1mm}
\alpha,\beta\in\mathbb{C}^*.$
\vspace{2mm}

\noindent Let us consider the family of automorphisms $\varphi=\varphi_{\varepsilon}=(\varepsilon^{d+1}x,\varepsilon y).$ We have
\begin{small}
\begin{align*}
\frac{1}{\varepsilon^{d+1}}\varphi^*\omega=\alpha\mathrm{d}x+
\Big(\varepsilon^{d}x\widehat{B}_{d-1}(\varepsilon^{d}x,y)+\beta y^{d}\Big)\mathrm{d}y+
\varepsilon^{d+1}C_{d}(\varepsilon^{d}x,y)\left(x\mathrm{d}y-y\mathrm{d}x\right)+
\sum_{i=1}^{d}\varepsilon^{i}A_i(\varepsilon^{d}x,y)\mathrm{d}x+
x\sum_{i=1}^{d-1}\varepsilon^{i}\widetilde{B}_{i-1}(\varepsilon^{d}x,y)\mathrm{d}y
\end{align*}
\end{small}
\hspace{-1mm}which tends to  $\alpha\mathrm{d}x+\beta y^{d}\mathrm{d}y$ as $\varepsilon$ tends to $0.$ Clearly $\alpha\mathrm{d}x+\beta y^{d}\mathrm{d}y$ defines a foliation conjugated to  $\F_{2}^{d}$; as a result $\F$ degenerates onto $\F_{2}^{d}.$
\end{proof}

\begin{eg}\label{eg:H2}
Let us consider the homogeneous foliation $\mathcal{H}_{2}^{d}$ defined in the affine chart  $z=1$ by the $1$-form
$$\omegaoverline_{2}^{d}=x^d\mathrm{d}x-y^d\mathrm{d}y.$$
We know from~\cite[Proposition~4.1]{BM18} that $\mathcal{H}_{2}^{d}$ has no non-degenerate singularity with  \textsc{Baum-Bott} index~$4$ and that
\[
\hspace{4.34cm}\mathrm{Flex}(\mathcal{H}_{2}^{d},d-1)=\{xy=0\}\setminus\{[0:0:1]\}\neq\vide.
\]
Thus $\mathcal{H}_{2}^{d}$ degenerates onto~$\F_{2}^{d}$~(Proposition~\ref{pro:cond-nécess-suff-dgnr-F2}) and it does not degenerate onto $\F_{1}^{d}$~(Proposition~\ref{pro:cond-nécess-suff-dgnr-F1}). Consequently, according to Remark~\ref{rem:adherence-orbite-homogene}, we have
\[
\hspace{1.5cm}\overline{\mathcal{O}(\mathcal{H}_{2}^{d})}=\mathcal{O}(\mathcal{H}_{2}^{d})\cup\mathcal{O}(\F_{2}^{d}).
\]
\end{eg}

\begin{eg}[\textsc{Jouanolou}'s foliation]\label{eg:Jouanolou}
Let us consider the foliation $\F_{J}^{d}$ of degree $d\geq2$ on $\pp$ defined, in the affine chart~$z=1,$ by $$\omega_{J}^{d}=(x^{d}y-1)\mathrm{d}x+(y^{d}-x^{d+1})\mathrm{d}y.$$
This example is due to \textsc{Jouanolou} and is historically the first explicit example of foliation without invariant algebraic curve (\cite{Jou79}). The point $p=(0,0)$ is an inflection point of maximal order $d-1$ of $\F_{J}^{d}$ because $\mathrm{T}^{\mathbb{P}}_{p}\F_{J}^{d}=\{x=0\}\hspace{2mm} \text{and} \hspace{2mm}\omega_{J}^{d}\wedge\mathrm{d}x\Big|_{x=0}=y^d\mathrm{d}y\wedge\mathrm{d}x.$ As a result $\F_{J}^{d}$ degenerates onto~$\F_{2}^{d}$ (Proposition~\ref{pro:cond-nécess-suff-dgnr-F2}). However,~we know from~\cite[Section~3]{LNP06} that every singularity  $s$ of $\F_{J}^{d}$ is non-degenerate with  \textsc{Baum-Bott} index $$\mathrm{BB}(\F_{J}^{d},s)=\frac{(d+2)^2}{d^2+d+1}\neq4,$$ so that $\F_{J}^{d}$ does not degenerate onto $\F_{1}^{d}$ (Proposition~\ref{pro:cond-nécess-suff-dgnr-F1}).
\end{eg}

\noindent The converse of assertion~\textbf{\textit{2.}} of Proposition~\ref{pro:cond-nécess-suff-dgnr-F2} is false as the following example shows.
\begin{eg}\label{eg:sans-inflex-maximal}
Let $\F$ be the foliation of degree $d\geq2$ on $\pp$ defined in the affine chart $z=1$ by
\begin{align*}
\hspace{3.3cm}&\omega=\mathrm{d}x+P(y)\mathrm{d}y,&& \text{where }P\in\C[y],\hspace{1mm}\deg P=d.
\end{align*}
It is easy to check that  $\Sing(\F)=\big\{[1:0:0]\big\}$\, and \,$\ItrF=P'(y).$  If the derivative  $P'$ has a single root, {\it i.e} if~$P$~is of the form $P(y)=a(y-\alpha)^d+b,$ where $\alpha,a,b\in\C,a\neq0,$ then $\F$ is conjugated to $\F_{2}^{d}$; indeed, we have
\begin{align*}
\hspace{1.65cm}&\frac{1}{a}\varphi^*\omega=\mathrm{d}x+y^{d}\mathrm{d}y, \quad \text{where}\hspace{1mm}\varphi=(ax-by,y+\alpha).
\end{align*}
\noindent We assume that the derivative $ P'$ has at least two distinct roots; this implies that $d\geq 3.$ A straightforward computation shows that $\F$ has no inflection point of maximal order $d-1$,  {\it i.e.} $\mathrm{Flex}(\F,d-1)=\vide.$ However, $\F$ degenerates onto $\F_{2}^{d}$. Indeed, by writing~$P(y)=a_{0}+a_{1}y+\cdots+a_{d}y^d$, $a_{i}\in\mathbb{C},\,a_{d}\neq0,$ and by putting $\psi=$\begin{small}$\left(\dfrac{a_{d}}{\varepsilon^{d+1}}x,\dfrac{1}{\varepsilon}y\right)$\end{small}, we obtain that
\begin{align*}
\hspace{-0.64cm}&\lim_{\varepsilon\to 0}\frac{\varepsilon^{d+1}}{a_{d}}\psi^*\omega=\mathrm{d}x+y^{d}\mathrm{d}y.
\end{align*}
\end{eg}

\begin{ques}\label{ques:2}
{\sl
Let $\F$ be a foliation of degree  $d\geq3$ on $\pp.$
Is it true that if $\F$ degenerates onto $\F_2^d$ then
\begin{itemize}
\item [--] either $\F$ possesses an inflection point of maximal order $d-1$,
\item [--] or $\F$ is conjugated to Example~\ref{eg:sans-inflex-maximal}, {\it i.e.} up to linear conjugation $\F$ is given by a $1$-form of type $\mathrm{d}x+P(y)\mathrm{d}y$ with $P\in\C[y],$ $\deg P=d$?
\end{itemize}
}
\end{ques}

\begin{pro}\label{pro:ouvert-U2}
{\sl
Let $d$ be an integer greater than or equal to $2.$ Let us denote by  $U_2(d)$ the set of foliations $\F\in\mathbf{F}(d)$
whose inflection divisor $\IF$ is transverse ({\it i.e.} $\IF=\ItrF$) and reduced. Then
\begin{itemize}
\item [\textit{(i)}] $U_2(d)$ contains the \textsc{Jouanolou}'s foliation $\F_{J}^{d}$ and it is a (non-empty) \textsc{Zariski} open subset of~$\mathbf{F}(d);$
\smallskip
\item [\textit{(ii)}] for any $d\geq3,$ every foliation $\F\in U_2(d)$ has a finite number (possibly zero) of transverse inflection points of order greater than or equal to $2$; in other words, the set $\bigcup\limits_{k=3}^{d}\mathrm{Flex}(\F,k-1)$ is finite.
\end{itemize}
}
\end{pro}

\noindent To establish this proposition, let us first prove the following lemma.
\begin{lem}\label{lem:Tang-F-TpF-p}
{\sl
Let $\F$ be a foliation of degree $d\geq2$ on $\pp,$ $p$ a regular point of $\F$ and $\mathrm{X}$ a polynomial vector field defining $\F$ in an affine chart $(x,y)$ containing $p.$ Then, for any $k\in\{2,3,\ldots,d\},$ $\Tang(\F,\mathrm{T}^{\mathbb{P}}_{p}\F,p)\geq k$ if and only if the matrix
$\left(\begin{array}{cccc}
\mathrm{X}(x) & \mathrm{X}^2(x) & \cdots & \mathrm{X}^k(x)\\
\mathrm{X}(y) & \mathrm{X}^2(y) & \cdots & \mathrm{X}^k(y)
\end{array}\right)\Big|_p$
has rank $1.$
}
\end{lem}

\begin{rem}
If $\mathrm{X}=\sum\limits_{i=1}^{n}\mathrm{X}_i(z_1,\ldots,z_n)\dfrac{\raisebox{-0.5mm}{$\partial$}}{\partial z_i}$ is a holomorphic vector field on  $\C^n$ and if $t\mapsto\alpha(t)$ is an integral curve of $\mathrm{X},$ then we have the following formula which can be easily proved by induction on~$j:$
\begin{align}\label{equa:derivees-courbe-integrale}
\frac{\raisebox{-0.8mm}{$\mathrm{d}^j$}}{\mathrm{d}\hspace{0.2mm}t^j}\alpha(t)=(\mathrm{X}^j(z_1),\ldots,\mathrm{X}^j(z_n))\circ\alpha(t).
\end{align}
\end{rem}

\begin{proof}
Let $t\mapsto\alpha(t)$ be the integral curve of $\mathrm{X}$ passing through $p$ at $t=0.$ The point~$p$ being regular~for~$\F,$~we have $\mathrm{T}_{\hspace{-0.4mm}p}\F\ni\alpha'(0)=\mathrm{X}(p)\neq0.$ Up to linear conjugation, we can assume that  $p=(0,0)$ and~$\mathrm{T}^{\mathbb{P}}_{p}\F=\{y=0\}.$ We can then write $\alpha(t)=\left(\sum\limits_{i\ge1}x_i\frac{t^i}{i!},\sum\limits_{i\ge1}y_i\frac{t^i}{i!}\right)$ with $y_1=0$ and $x_1\neq0.$ Thus,~$\Tang(\F,\mathrm{T}^{\mathbb{P}}_{p}\F,p)=\nu(g(t),0),$ where $g(t)=\sum\limits_{i\ge2}y_i\frac{t^i}{i!}.$ As a result, $\Tang(\F,\mathrm{T}^{\mathbb{P}}_{p}\F,p)\geq k$ if and only if $y_2=y_3=\cdots=y_k=0$, or equivalently if and only if the matrix
$
\left(\begin{array}{cccc}
x_1 & x_2 & \cdots & x_k\\
0   & y_2 & \cdots & y_k
\end{array}\right)
$
has rank~$1.$ Now,~by~using formula~(\ref{equa:derivees-courbe-integrale}), we see that
\[
\left(\begin{array}{cccc}
x_1 & x_2 & \cdots & x_k\\
0   & y_2 & \cdots & y_k
\end{array}\right)
=
\left(\begin{array}{cccc}
\mathrm{X}(x) & \mathrm{X}^2(x) & \cdots & \mathrm{X}^k(x)\\
\mathrm{X}(y) & \mathrm{X}^2(y) & \cdots & \mathrm{X}^k(y)
\end{array}\right)\Big|_{(x,y)=(0,0)},
\]
hence the lemma follows.
\end{proof}

\begin{proof}[\sl Proof of Proposition~\ref{pro:ouvert-U2}]
\textit{(i)} For $\F\in\mathbf{F}(d),$ to say that $\IF$ is transverse and reduced means that $\F$ has no invariant line and that $\IF$ has no multiple component, which shows that $U_2(d)$ is a \textsc{Zariski} open subset of $\mathbf{F}(d).$

\noindent As we have already mentioned in Example~\ref{eg:Jouanolou}, the \textsc{Jouanolou}'s foliation $\F_{J}^{d}$ has no invariant algebraic curve~\cite{Jou79}; in particular, it has no invariant line and consequently $\mathrm{I}_{\F_{J}^{d}}=\mathrm{I}_{\F_{J}^{d}}^{\hspace{0.2mm}\mathrm{tr}}.$ To establish the first announced assertion, it remains to prove that $\mathrm{I}_{\F_{J}^{d}}$ is reduced. In homogeneous coordinates, the foliation $\F_{J}^{d}$ is defined by the vector field $y^d\frac{\partial}{\partial x}+z^d\frac{\partial}{\partial y}+x^d\frac{\partial}{\partial z}$; an immediate computation, using formula~(\ref{equa:ext1}), shows that $\mathrm{I}_{\F_{J}^{d}}$ has equation  $F(x,y,z)=0,$ where
\[
F(x,y,z)=x^{2d+1}z^{d-1}+y^{2d+1}x^{d-1}+z^{2d+1}y^{d-1}-3x^dy^dz^d.
\]

\noindent We must show that $F$ has no multiple factor in $\C[x,y,z]$. Since $F\in\Z[x,y,z],$ it suffices to show that  $F$ has no multiple factor in $\mathbb{F}_{2}[x,y,z].$ Indeed, if $F$  had a multiple factor in $\C[x,y,z],$~then one of the resultants $\mathrm{Res}_{x}(F,\frac{\partial F}{\partial x})\in\Z[y,z]\hspace{0.4mm}$ or $\hspace{0.4mm}\mathrm{Res}_{y}(F,\frac{\partial F}{\partial y})\in\Z[x,z]\hspace{0.4mm}$ or $\hspace{0.4mm}\mathrm{Res}_{z}(F,\frac{\partial F}{\partial z})\in\Z[x,y]$ would be identically zero and therefore so would be its reduction modulo $2$; so that $F$ would also have a multiple factor in $\mathbb{F}_{\mathrm{2}}[x,y,z].$ We have to show that $\mathrm{gcd}(F,\frac{\partial F}{\partial x},\frac{\partial F}{\partial y},\frac{\partial F}{\partial z})=1$ in $\mathbb{F}_{2}[x,y,z],$ or equivalently that
\begin{align*}
&
\mathrm{gcd}(F,\tfrac{\partial F}{\partial x})=1\hspace{1mm}\text{in}\hspace{1mm}\mathbb{F}_{2}(y,z)[x],&&
\mathrm{gcd}(F,\tfrac{\partial F}{\partial y})=1\hspace{1mm}\text{in}\hspace{1mm}\mathbb{F}_{2}(x,z)[y],&&
\mathrm{gcd}(F,\tfrac{\partial F}{\partial z})=1\hspace{1mm}\text{in}\hspace{1mm}\mathbb{F}_{2}(x,y)[z].
\end{align*}
The coordinates  $x,y,z$ playing a symmetric role, it suffices again to show that $\mathrm{gcd}(F,\tfrac{\partial F}{\partial x})=1$ in $\mathbb{F}_{2}(y,z)[x].$ In~$\mathbb{F}_{2}[x,y,z]$ we have
\begin{align*}
&
F=x^{2d+1}z^{d-1}+y^{2d+1}x^{d-1}+z^{2d+1}y^{d-1}+x^dy^dz^d
&&\hspace{2mm}\text{and}\hspace{2mm}&&
\tfrac{\partial F}{\partial x}=x^{d-2}\left(x^{d+2}z^{d-1}+dxy^dz^d+(d+1)y^{2d+1}\right).
\end{align*}
Then $x=0$ is not a root of  $F\in\mathbb{F}_{2}(y,z)[x]$ and consequently
\begin{align*}
&
\mathbb{F}_{2}(y,z)[x]\ni\gcd(F,\tfrac{\partial F}{\partial x})=\gcd(F,\varphi),
\quad\text{where}\quad
\varphi=x^{d+2}+dxzy^d+(d+1)\frac{y^{2d+1}}{z^{d-1}}.
\end{align*}
Moreover, a straightforward computation shows that
\begin{Small}
\begin{align*}
x^3F=\left(x^{d+2}z^{d-1}-(d-1)xy^dz^d-dy^{2d+1}\right)\varphi
+y^{d-1}z^{2d+1}\left(x+\frac{y^{d+1}}{z^d}\right)\left(x^2+(d^2-d-1)\frac{y^{d+1}}{z^d}x+d(d+1)\frac{y^{2d+2}}{z^{2d}}\right),
\end{align*}
\end{Small}
\hspace{-1mm}so that
\begin{align*}
\mathbb{F}_{2}(y,z)[x]\ni\gcd(F,\varphi)&=
\gcd\left(\Big(x+\frac{y^{d+1}}{z^d}\Big)\Big(x-\frac{y^{d+1}}{z^d}\Big),\,\varphi\right),
\hspace{1mm}\text{because}\hspace{1mm}
d^2-d\equiv d(d+1)\equiv0\hspace{-2.5mm}\mod 2
\\
&=\gcd\left(x-\frac{y^{d+1}}{z^d},\,x^{d+2}+dxzy^d+(d+1)\frac{y^{2d+1}}{z^{d-1}}\right)
\\
&=\gcd\left(x-\frac{y^{d+1}}{z^d},\,x^{d+2}-\frac{y^{2d+1}}{z^{d-1}}\right)
\\
&=1,
\end{align*}
because $\left(\dfrac{y^{d+1}}{z^d}\right)^{d+2}\neq\dfrac{y^{2d+1}}{z^{d-1}}$ in the field $\mathbb{F}_{2}(y,z).$ As a result $\mathbb{F}_{2}(y,z)[x]\ni\gcd(F,\tfrac{\partial F}{\partial x})=1.$
\medskip

\textit{(ii)} Let $\F$ be a foliation of degree $d\geq3$ on $\pp$ with reduced and transverse inflection divisor  $\mathrm{I}_{\F}$, {\it i.e.} $\F\in U_2(d).$ We want to show that the set $\Gamma(\F):=\bigcup\limits_{k=3}^{d}\mathrm{Flex}(\F,k-1)$ is finite. By definition of~$\Gamma(\F)$ we have
\begin{align}\label{equa:Gamma-F}
\Gamma(\F)\subset\Big\{p\in\pp\hspace{1mm}\vert\hspace{1mm} p\not\in\Sing(\F),\hspace{1mm} \Tang(\F,\mathrm{T}^{\mathbb{P}}_{p}\F,p)\geq3\Big\}.
\end{align}
Let $\mathrm{X}$ be a vector field defining $\F$ in an affine chart $\C^2=\{(x,y)\}\subset\pp.$ Lemma~\ref{lem:Tang-F-TpF-p} and inclusion~(\ref{equa:Gamma-F}) imply that $\Gamma(\F)\cap\C^2$ is contained in the set of points $p\in\C^2$ such that
\begin{align*}
&
\left(
\begin{array}{c}
\mathrm{X}(x)\\
\mathrm{X}(y)
\end{array}
\right)(p)\neq
\left(
\begin{array}{c}
0\\
0
\end{array}
\right),
&&
\mathrm{I_X}(p):=\left|
\begin{array}{cc}
\mathrm{X}(x) & \mathrm{X}^2(x)\\
\mathrm{X}(y) & \mathrm{X}^2(y)
\end{array}
\right|(p)=0,
&&
\mathrm{X}(\mathrm{I_X})(p)=\left|
\begin{array}{cc}
\mathrm{X}(x) & \mathrm{X}^3(x)\\
\mathrm{X}(y) & \mathrm{X}^3(y)
\end{array}
\right|(p)=0.
\end{align*}
Now, the affine chart $\C^2=\{(x,y)\}\subset\pp$ being arbitrary, $\Gamma(\F)$ is finite if and only if $\Gamma(\F)\cap\C^2$ is finite. It~suffices therefore to show that the algebraic curves  $\mathrm{I}_{\F}\cap\C^2=\{\mathrm{I_X}(x,y)=0\}$ and $\mathcal{C}:=\{\mathrm{X}(\mathrm{I}_\mathrm{X})(x,y)=0\}$ intersect at a finite number of points, {\it i.e.} that they have no common component. Let us argue by contradiction and assume that there exist $K,L,L'\in\C[x,y],$ with $\deg K>0,$ such that $\mathrm{I_X}=KL$ and $\mathrm{X}(\mathrm{I}_\mathrm{X})=KL'.$ Then $KL'=\mathrm{X}(KL)=\mathrm{X}(K)L+K\mathrm{X}(L)$ and therefore $\mathrm{X}(K)L=K(L'-\mathrm{X}(L)).$ Moreover, the hypothesis that $\mathrm{I}_{\F}$ is reduced implies that $\mathrm{gcd}(K,L)=1.$ It follows that there is  $L''\in\C[x,y]$ such that $\mathrm{X}(K)=KL''$, which means that the algebraic curve $\mathcal{C}':=\{K(x,y)=0\},$ contained in~$\mathrm{I}_{\F}$, is invariant by $\F,$ contradicting the hypothesis that $\mathrm{I}_{\F}$ is transverse.
\end{proof}

\begin{thm}\label{thm:bassin-attraction-F2}
{\sl Let $d$ be an integer greater than or equal to $2.$ Let us denote by  $\Sigma_2(d)$ the subset of~$\mathbf{F}(d)$ defined by
\begin{align*}
&\Sigma_2(d):=\Big\{\F\in \mathbf{F}(d)\hspace{1mm}\vert\hspace{1mm}\mathrm{Flex}(\F,d-1)\neq\vide\Big\}.
\end{align*}
Then
\begin{itemize}
\item [\textit{(a)}] $\mathbf{B}(\F_{2}^{2})=\mathbf{F}(2)\setminus\mathbf{FC}(2)=\Sigma_2(2)$ and, for any $d\geq3$, we have $\emptyset\neq \Sigma_2(d)\varsubsetneq \mathbf{B}(\F_{2}^{d});$
\smallskip
\item [\textit{(b)}] $\Sigma_2(d)$ is a constructible subset of~$\mathbf{F}(d)$;
\smallskip
\item [\textit{(c)}] for any $d\geq3$, we have $\dim\Sigma_2(d)\geq\dim \mathbf{F}(d)-(d-3).$
\end{itemize}
In particular, the set $\Sigma_2(3)$, and therefore $\mathbf{B}(\F_{2}^{3}),$ contains a non-empty \textsc{Zariski} open subset of $\mathbf{F}(3).$
}
\end{thm}

\begin{proof}
\textit{(a)}
As we have already said in Introduction, the first equality $\mathbf{B}(\F_{2}^{2})=\mathbf{F}(2)\setminus\mathbf{FC}(2)$ follows from~\cite[Theorem~3]{CDGBM10}. The second equality $\mathbf{F}(2)\setminus\mathbf{FC}(2)=\Sigma_2(2)$ is a consequence of the following obvious remark: if $\F\in\mathbf{F}(2)\setminus\mathbf{FC}(2)$ then every transverse inflection point of $ \F $ is of order $1$.

\noindent The set $\Sigma_{2}(d)$ contains the foliations $\mathcal{H}_{2}^{d}$ and $\F_J^d$ (Examples~\ref{eg:H2}~and~\ref{eg:Jouanolou}) and is therefore non-empty. According to assertion~\textbf{\textit{2.}} of Proposition~\ref{pro:cond-nécess-suff-dgnr-F2}, we have $\Sigma_{2}(d)\subset\mathbf{B}(\F_{2}^{d})$; this inclusion is strict for any $d\geq3$ as Example~\ref{eg:sans-inflex-maximal} shows.
\medskip

\textit{(b)} Let~$\pi\hspace{1mm}\colon\mathbf{F}(d)\times\pp\to\mathbf{F}(d)$ be the projection onto the first factor; notice that  $\Sigma_{2}(d)=\pi(W_{2}(d)),$ where
\begin{align*}
W_{2}(d):&=\bigcup_{\F\in\Sigma_{2}(d)}\{\F\}\times\mathrm{Flex}(\F,d-1)\\
&=\Big\{(\F,p)\in\mathbf{F}(d)\times\pp\hspace{1mm}\vert\hspace{1mm} p\not\in\Sing(\F),\hspace{1mm} \Tang(\F,\mathrm{T}^{\mathbb{P}}_{p}\F,p)=d\Big\}.
\end{align*}
By Lemma~\ref{lem:Tang-F-TpF-p}, $W_{2}(d)$ can be rewritten as
\begin{align}\label{equa:W2(d)}
W_{2}(d)=\Big\{
(\F,p)\in\mathbf{F}(d)\times\pp\hspace{1mm}\vert\hspace{1mm}
\left(
\begin{array}{c}
\mathrm{X}(x)\\
\mathrm{X}(y)
\end{array}
\right)(p)\neq
\left(
\begin{array}{c}
0\\
0
\end{array}
\right),
\hspace{1mm}
\left|
\begin{array}{cc}
\mathrm{X}(x) & \mathrm{X}^j(x)\\
\mathrm{X}(y) & \mathrm{X}^j(y)
\end{array}
\right|(p)=0,
j=2,\ldots,d
\Big\},
\end{align}
where $\mathrm{X}$ denotes a polynomial vector field defining $\F$ in an affine chart $(x,y)$ containing $p.$ It follows that $W_{2}(d)$ is a quasi-projective subvariety of $\mathbf{F}(d)\times\pp.$ Therefore, by \textsc{Chevalley}'s~theorem \cite[Exercise~II.3.19]{Har77}, the set $\Sigma_{2}(d)=\pi(W_{2}(d))$ is constructible.
\medskip

\textit{(c)} From the above discussion and Proposition~\ref{pro:ouvert-U2}~\textit{(i)}, we have $\F_J^d\in U_2(d)\cap\Sigma_{2}(d)\neq\vide$ ($U_2(d)$~being the set of foliations of $\mathbf{F}(d)$ with reduced and transverse inflection divisor). Therefore there exists an irreducible component $\Sigma_{2}^{0}(d)$ of $\Sigma_{2}(d)$ such that $U_2(d)\cap\Sigma_{2}^{0}(d)\neq\vide.$ We denote by $\pi_0\hspace{1mm}\colon W_{2}(d)\to\mathbf{F}(d)$ the restriction of~$\pi$ to~$W_{2}(d).$ Let $W_{2}(d)=\bigcup\limits_{i=1}^{n}W_{2}^{i}(d)$ be the decomposition of $W_{2}(d)$ into its irreducible components. Then, by arguing as in the proof of Theorem~\ref{thm:bassin-attraction-F1}, we see that there is $k\in\{1,\ldots,n\}$ such that~$\overline{\pi_0(W_{2}^{k}(d))}=\overline{\Sigma_{2}^{0}(d)}.$ Since $U_2(d)$ is a \textsc{Zariski} open subset of $\mathbf{F}(d)$ (Proposition~\ref{pro:ouvert-U2}~\textit{(i)}), the morphism $\pi_0$ therefore induces by~restriction a dominant morphism of quasi-projective varieties $\pi_{0}^{k}\hspace{1mm}\colon W_{2}^{k}(d)\cap\pi_{0}^{-1}(U_2(d))\to\overline{\Sigma_{2}^{0}(d)}\cap U_2(d).$

\noindent  Notice that, for any $\F\in U_2(d)\cap\Sigma_{2}(d),$ the fiber $\pi_{0}^{-1}(\F)$ is finite and non-empty, because~$\pi_{0}^{-1}(\F)=\{\F\}\times\mathrm{Flex}(\F,d-1)$ and $\mathrm{Flex}(\F,d-1)$ is finite and non-empty by~assertion~\textit{(ii)}~of~Proposition~\ref{pro:ouvert-U2}. Since~$\pi_{0}(W_{2}^{k}(d)\cap\pi_{0}^{-1}(U_2(d)))\subset U_2(d)\cap\Sigma_{2}(d),$ we deduce that all the non-empty fibers of  $\pi_{0}^{k}$ are finite~and therefore zero-dimensional. The fiber dimension theorem (\emph{cf.}~\cite[Theorem~3, page 49]{Mum88}) then ensures that~$\dim (W_{2}^{k}(d)\cap\pi_{0}^{-1}(U_2(d)))=\dim (\overline{\Sigma_{2}^{0}(d)}\cap U_2(d))$; since $W_{2}^{k}(d)\cap\pi_{0}^{-1}(U_2(d))$ and $\overline{\Sigma_{2}^{0}(d)}\cap U_2(d)$ are non-empty open subsets of the irreducible varieties $W_{2}^{k}(d)$ and~$\overline{\Sigma_{2}^{0}(d)}$ respectively, we have
\begin{align*}
\hspace{2.3cm}\dim\overline{\Sigma_{2}^{0}(d)}=\dim (\overline{\Sigma_{2}^{0}(d)}\cap U_2(d))=\dim (W_{2}^{k}(d)\cap\pi_{0}^{-1}(U_2(d)))=\dim W_{2}^{k}(d).
\end{align*}
Now, it follows from~(\ref{equa:W2(d)}) that each irreducible component $W_{2}^{i}(d)$ of $W_{2}(d)$ has dimension
\[\dim W_{2}^{i}(d)\geq \dim(\mathbf{F}(d)\times\pp)-(d-1)=\dim\mathbf{F}(d)-(d-3),\]
hence
\[\hspace{1.55cm}\dim\Sigma_{2}(d)=\dim\overline{\Sigma_{2}(d)}\geq\dim\overline{\Sigma_{2}^{0}(d)}=\dim W_{2}^{k}(d)\geq\dim\mathbf{F}(d)-(d-3).\]

\noindent The subset $\Sigma_{2}(d)\subset\mathbf{F}(d)$ being constructible, it contains a dense open subset of its closure~$\overline{\Sigma_{2}(d)}$. In degree $d=3$ we have $\dim\overline{\Sigma_{2}(3)}\geq\dim\mathbf{F}(3)$ and therefore $\dim\overline{\Sigma_{2}(3)}=\dim\mathbf{F}(3),$ so that $\overline{\Sigma_{2}(3)}=\mathbf{F}(3)$ because $\mathbf{F}(3)$ is irreducible. It follows that $\Sigma_{2}(3)$ contains a dense open subset of $\mathbf{F}(3)$. This ends the proof of~the~theorem.
\end{proof}
\smallskip

\begin{rem}\label{rem:degenerescence-F1-F2}
The set $\mathbf{F}(d)$ contains elements which degenerate onto both  $\F_{1}^{d}$ and $\F_{2}^{d},$ {\it e.g.} the family of foliations $\mathcal{G}^{d}(\gamma)$, $\gamma\in\C.$ Indeed, on the one hand, we have seen (Example~\ref{eg:famille-G-gamma}) that $\mathcal{G}^{d}(\gamma)$~degenerates~onto~$\F_1^d.$ On the other hand, by putting $\varphi=(\frac{x}{\varepsilon},\frac{y}{\varepsilon})$ we obtain that $\lim\limits_{\varepsilon\to 0}\varepsilon^{d+1}\varphi^*\eta^{d}(\gamma)=\omegaoverline_{2}^{d},$ which shows that  $\mathcal{G}^{d}(\gamma)$ degenerates onto the homogeneous foliation $\mathcal{H}_2^d$ (Example~\ref{eg:H2}) and therefore, by transitivity, onto $\F_2^d.$
\end{rem}

\begin{eg}\label{eg:H12}
Let us consider the homogeneous foliation  $\mathcal{H}_{1,2}^{d}$ defined in the affine chart $z=1$ by the $1$-form
$$\hspace{-1.1cm}\omegaoverline_{1,2}^{d}=(x^d+y^d)\mathrm{d}x+x^d\mathrm{d}y.$$ This foliation degenerates onto both $\F_{1}^{d}$ and $\F_{2}^{d}.$ Indeed, on the one hand, $\mathcal{H}_{1,2}^{d}$ is given in the affine chart~$y=1$ by $$\hspace{1.05cm}\thetaoverline_{1,2}^{d}=x\mathrm{d}z-z\mathrm{d}x+x^d\mathrm{d}z+x^d(x\mathrm{d}z-z\mathrm{d}x);$$ we see that the  point $[0:1:0]$ is a radial singularity of maximal order $d-1$ of $\mathcal{H}_{1,2}^{d}.$ Thus, by Proposition~\ref{pro:cond-nécess-suff-dgnr-F1}, $\mathcal{H}_{1,2}^{d}$ degenerates onto~$\F_{1}^{d}.$ On the other hand, a straightforward computation shows that $$\hspace{2cm}\mathrm{Flex}(\mathcal{H}_{1,2}^{d},d-1)=\{y=0\}\setminus\{[0:0:1]\}\neq\vide;$$ consequently, $\mathcal{H}_{1,2}^{d}$~also degenerates onto $\F_{2}^{d}$~(Proposition~\ref{pro:cond-nécess-suff-dgnr-F2}).

\noindent Since $\overline{\mathcal{O}(\mathcal{H}_{1,2}^{d})}\subset\mathcal{O}(\mathcal{H}_{1,2}^{d})\cup\mathcal{O}(\F_{1}^{d})\cup\mathcal{O}(\F_{2}^{d})$ (Remark~\ref{rem:adherence-orbite-homogene}), we deduce that in fact
\[
\hspace{0.7cm}\overline{\mathcal{O}(\mathcal{H}_{1,2}^{d})}=\mathcal{O}(\mathcal{H}_{1,2}^{d})\cup\mathcal{O}(\F_{1}^{d})\cup\mathcal{O}(\F_{2}^{d}).
\]
\end{eg}

\begin{thm}\label{thm:intersection-bassin-attraction-F1-F2}
{\sl Let $d$  be an integer greater than or equal to $2.$ Then
\begin{itemize}
\item [\textit{(a)}] $\emptyset\neq\Sigma_1(d)\cap\Sigma_2(d)\subset\mathbf{B}(\F_{1}^{d})\cap\mathbf{B}(\F_{2}^{d})\supset\mathbf{B}(\mathcal{H}_{1,2}^{d});$
\smallskip
\item [\textit{(b)}] $\mathbf{B}(\mathcal{H}_{1,2}^{d})$ contains a quasi-projective subvariety of $\mathbf{F}(d)$ of dimension equal to $\dim\mathbf{F}(d)-3d.$
\end{itemize}
}
\end{thm}

\begin{proof}
\textit{(a)} The intersection $\Sigma_1(d)\cap\Sigma_2(d)$ contains the homogeneous foliation $\mathcal{H}_{1,2}^{d}$ (Example~\ref{eg:H12}) and is therefore non-empty. The inclusion $\Sigma_1(d)\cap\Sigma_2(d)\subset\mathbf{B}(\F_{1}^{d})\cap\mathbf{B}(\F_{2}^{d})$ follows from Theorems~\ref{thm:bassin-attraction-F1}~and~\ref{thm:bassin-attraction-F2}.

\noindent Let us show the inclusion $\mathbf{B}(\mathcal{H}_{1,2}^{d})\subset\mathbf{B}(\F_{1}^{d})\cap\mathbf{B}(\F_{2}^{d}).$ Let $\F\in\mathbf{B}(\mathcal{H}_{1,2}^{d})$, {\it i.e.} $\F\in\mathbf{F}(d)$ such that $\mathcal{H}_{1,2}^{d}\in\overline{\mathcal{O}(\F)}.$ Since $\mathcal{H}_{1,2}^{d}$ degenerates onto $\F_{i}^{d},i=1,2,$ it follows that $\F_{i}^{d}\in\overline{\mathcal{O}(\mathcal{H}_{1,2}^{d})}\subset\overline{\mathcal{O}(\F)}$, hence $\F\in\mathbf{B}(\F_{1}^{d})\cap\mathbf{B}(\F_{2}^{d}).$
\smallskip

\textit{(b)} Let us denote by $\Sigma(\mathcal{H}_{1,2}^{d})$ the subset of $\mathbf{F}(d)$ defined as follows: an element $\F$ of $\mathbf{F}(d)$ belongs to $\Sigma(\mathcal{H}_{1,2}^{d})$ if and only if
\smallskip
\begin{itemize}
\item [\textit{(1)}] $\F$ admits an invariant line  $\ell$;

\item [\textit{(2)}] there is a system of homogeneous coordinates $[x:y:z]\in\pp$ in which $\ell=\{z=0\}$ and $\F$ is defined in the affine chart  $z=1$ by a $1$-form $\omega$ of type
\begin{align*}
&\omega=\sum\limits_{i=0}^{d-1}\omega_i+\lambda\omegaoverline_{1,2}^{d}=\sum\limits_{i=0}^{d-1}\omega_i+\lambda\left((x^d+y^d)\mathrm{d}x+x^d\mathrm{d}y\right),
\end{align*}
where $\lambda\in\C^*$ and the $\omega_i$'s are homogeneous $1$-forms of degree $i.$
\end{itemize}
\smallskip

\noindent Notice that $\Sigma(\mathcal{H}_{1,2}^{d})\subset\mathbf{B}(\mathcal{H}_{1,2}^{d}).$ Indeed, by putting $\varphi=(\frac{x}{\varepsilon},\frac{y}{\varepsilon})$ and by writing $\omega_i=P_i(x,y)\mathrm{d}x+Q_i(x,y)\mathrm{d}y,$ where $P_i,Q_i\in\C[x,y]_i,$ we obtain
\[
\varepsilon^{d+1}\varphi^*\omega=\sum_{i=0}^{d-1}(\varepsilon^{d-i}P_i(x,y)\mathrm{d}x+\varepsilon^{d-i}Q_i(x,y)\mathrm{d}y)+\lambda\omegaoverline_{1,2}^{d}
\]
which tends to $\lambda\omegaoverline_{1,2}^{d}$ as $\varepsilon$ tends  to $0$. It follows that $\mathcal{H}_{1,2}^{d}\in\overline{\mathcal{O}(\F)}$ for any $\F\in\Sigma(\mathcal{H}_{1,2}^{d}),$ hence the inclusion $\Sigma(\mathcal{H}_{1,2}^{d})\subset\mathbf{B}(\mathcal{H}_{1,2}^{d})$ holds.

\noindent Moreover, every foliation $\F\in\mathbf{F}(d)$ is given in the affine chart $z=1$ by a $1$-form of type
\begin{align*}
\sum\limits_{i=0}^{d}\big(A_i(x,y)\mathrm{d}x+B_i(x,y)\mathrm{d}y\big)+C_d(x,y)(x\mathrm{d}y-y\mathrm{d}x),
\end{align*}
where $A_i,B_i\in\C[x,y]_i,C_d\in\C[x,y]_d$ with $\mathrm{gcd}\big(yC_d-\sum\limits_{i=0}^{d}A_i,xC_d+\sum\limits_{i=0}^{d}B_i\big)=1$. Condition \textit{(2)} is then equivalent to taking $C_d\equiv0,\,A_d(x,y)=\lambda(x^d+y^d),\,B_d(x,y)=\lambda\,x^d.$ Since the set of foliations of $\mathbf{F}(d)$ admitting an invariant line is a \textsc{Zariski} closed subset of $\mathbf{F}(d),$ we deduce that $\Sigma(\mathcal{H}_{1,2}^{d})$ is a quasi-projective~subvariety~of~$\mathbf{F}(d).$

\noindent Since $\omega$ and $\mu\omega$ define the same foliation if $\mu\neq0,$ and the choice of a line $\ell\subset\pp$
is equivalent to the choice of a point in $\pd,$ conditions \textit{(1)} and \textit{(2)} imply that
\[
\dim\Sigma(\mathcal{H}_{1,2}^{d})=2+2\sum_{i=0}^{d-1}(i+1)=d^2+d+2=\dim\mathbf{F}(d)-3d.
\]
\end{proof}

\section{A family of foliations of $\mathbf{F}(d)$ with orbits of dimension less than or equal to $7$}\label{sec:F0-lambda-degre-d}

In this section we will establish some properties of the family $(\F_{0}^{d}(\lambda))_{\lambda\in\C^*}$ of foliations of degree $d$ on $\pp$ defined in the affine chart $z=1$ by $$\omega_{0}^{d}(\lambda)=x\mathrm{d}y-\lambda y\mathrm{d}x+y^{d}\mathrm{d}y.$$
In homogeneous coordinates, $\F_{0}^{d}(\lambda)$ is given by $$\Omega_{0}^{d}(\lambda)=-\lambda yz^d\mathrm{d}x+z\left(xz^{d-1}+y^d\right)\mathrm{d}y+y\left((\lambda-1)xz^{d-1}-y^d\right)\mathrm{d}z.$$ Thus, the singular locus of $\F_{0}^{d}(\lambda)$ consists of the two  points $s_1=[0:0:1]$ and~$s_2=[1:0:0]$. The singularity $s_1$ is non-degenerate with   \textsc{Baum}-\textsc{Bott} index $\mathrm{BB}(\F_{0}^{d}(\lambda),s_1)=2+\lambda+\frac{1}{\lambda}$ and the singular point $s_2$ has maximal algebraic multiplicity $d.$ We see that for $\lambda=1$ the $1$-form $\Omega_{0}^{d}(1)$ writes in the affine chart $x=1$ as $$z^d\mathrm{d}y+y^d(z\mathrm{d}y-y\mathrm{\mathrm{d}}z);$$ we deduce that $\F_{0}^{d}(1)$ is conjugated to the foliation  $\F_{1}^{d}$ and is therefore convex.
\smallskip

\noindent In the sequel we assume that $\lambda\in\C\setminus\{0,1\}$. A direct computation, using formula~(\ref{equa:ext1}), leads to
\begin{align}\label{equa:inflex-F0-lambda-degre-d}
\hspace{2cm}&\mathrm{I}_{\F_{0}^{d}(\lambda)}^{\hspace{0.2mm}\mathrm{inv}}=yz^{2d-1}
&&\text{and}&&
\mathrm{I}_{\F_{0}^{d}(\lambda)}^{\hspace{0.2mm}\mathrm{tr}}=(\lambda-1)x-\big((d-1)\lambda+1\big)y^d;
\end{align}
it follows that, for any $\lambda\in\C\setminus\{0,1\}$, $\F_{0}^{d}(\lambda)$ is not convex.

\noindent A straightforward computation shows that the algebraic curve~$(1-\lambda\hspace{0.2mm}d)x+y^d=0$ is invariant by~$\F_{0}^{d}(\lambda).$ What~is~more, the rational $1$-form  $\eta_{0}^{d}(\lambda)=\dfrac{\omega_{0}^{d}(\lambda)}{y\left(\left(1-\lambda\hspace{0.2mm}d\right)x+y^d\right)}$ is closed. For $\lambda=\frac{1}{d}$ we note that  $\eta_{0}^{d}(\frac{1}{d})=\dfrac{\omega_{0}^{d}(\lambda)}{y^{d+1}}$ has as first integral $\dfrac{x}{dy^{d}}-\ln y;$ this allows to see that $\mathrm{Iso}(\F_{0}^{d}(\frac{1}{d}))$ is the group $\{(\alpha^{d}x,\alpha y)\hspace{1mm}\vert\hspace{1mm}\alpha\in\mathbb{C}^*\}.$ When~$\lambda\in\C\setminus\{0,1,\frac{1}{d}\}$ a straightforward computation shows that $\eta_{0}^{d}(\lambda)$ integrates into $$\lambda\ln\left(\left(1-\lambda\hspace{0.2mm}d\right)x+y^d\right)-\ln y,$$ which allows to verify that the isotropy group is here again $$\mathrm{Iso}(\F_{0}^{d}(\lambda))=\{(\alpha^{d}x,\alpha y)\hspace{1mm}\vert\hspace{1mm}\alpha\in\mathbb{C}^*\}.$$
It follows in particular that, for any $\lambda\in\C\setminus\{0,1\}$, $\mathcal{O}(\F_{0}^{d}(\lambda))$ has dimension~$7.$
\smallskip

\noindent Notice that two foliations $\F_{0}^{d}(\lambda)$ and $\F_{0}^{d}(\lambda')$ are conjugated if and only if $\lambda=\lambda'.$
\begin{pro}\label{pro:cond-nécess-suff-dgnr-F0-lambda}
{\sl Let $\lambda$ be a nonzero complex number. Let $\F$ be an element of $\mathbf{F}(d)$ such that~$\F_{0}^{d}(\lambda)\not\in\mathcal{O}(\F).$

\noindent\textbf{\textit{1.}} If $\F$ degenerates onto  $\F_{0}^{d}(\lambda)$, then $\F$ admits a non-degenerate singular point $m$ satisfying $\mathrm{BB}(\F,m)=2+\lambda+\frac{1}{\lambda}$.

\noindent\textbf{\textit{2.}} If $\F$ possesses a non-degenerate singular point $m$ such that $$\mathrm{BB}(\F,m)=2+\lambda+\frac{1}{\lambda}\qquad \text{and}\qquad \kappa(\F,m)=d,$$
then $\F$ degenerates onto $\F_{0}^{d}(\lambda).$
}
\end{pro}

\begin{proof}
It suffices to argue as in the proof of Proposition~\ref{pro:cond-nécess-suff-dgnr-F1}, replacing the foliation $\F_{1}^{d}$ by $\F_{0}^{d}(\lambda)$ and the equality  $\mathrm{BB}(\F,m)=4$ by $\mathrm{BB}(\F,m)=2+\lambda+\frac{1}{\lambda}.$
\end{proof}

\begin{pro}\label{pro:Orbite-F0-lambda-degre-3-4-5-fermee}
{\sl The orbit $\mathcal{O}\big(\F_{0}^{d}(\lambda)\big)$ is closed in  $\mathbf{F}(d)$ in the following two cases:
\begin{itemize}
\item [\textit{(i)}] $d\geq3$\, and \,$\lambda=-\dfrac{1}{d-1};$
\smallskip
\item [\textit{(ii)}] $d\in\{3,4,5\}$\, and \,$\lambda\in\C^*.$
\end{itemize}
}
\end{pro}

\noindent The proof of this proposition uses the following lemma.
\begin{lem}\label{lem:Caracterisation-fermeture-Orbite-F0-lambda-degre-d}
{\sl
Let $\lambda$ be a nonzero complex number. Then, the orbit $\mathcal{O}\big(\F_{0}^{d}(\lambda)\big)$ is closed in  $\mathbf{F}(d)$ if and only if $\F_{0}^{d}(\lambda)$ does not degenerate onto $\F_{2}^{d}.$
}
\end{lem}

\begin{proof}
The direct implication is obvious. Let us prove the converse. From the above discussion, $\F_{0}^{d}(1)$ is~conjugated to the convex foliation  $\F_{1}^{d}$; therefore its orbit $\mathcal{O}\big(\F_{0}^{d}(1))$ is closed in $\mathbf{F}(d).$ For~any~$\lambda\in\C\setminus\{0,1\}$, the unique non-degenerate singular point $s_1=[0:0:1]$ of $\F_{0}^{d}(\lambda)$ has \textsc{Baum}-\textsc{Bott} index $\mathrm{BB}(\F_{0}^{d}(\lambda),s_1)=2+\lambda+\frac{1}{\lambda}\neq4$; this implies, according to assertion \textbf{\textit{1.}} of Proposition~\ref{pro:cond-nécess-suff-dgnr-F1}, that $\F_{0}^{d}(\lambda)$ does not degenerate onto~$\F_{1}^{d}.$ Moreover, for any $\lambda\in\C\setminus\{0,1\}$, $\mathcal{O}(\F_{0}^{d}(\lambda))$ has dimension~$7.$ The converse implication then follows immediately from~Corollary~\ref{coralph:dim-o(F)=<7}.
\end{proof}

\begin{proof}[\sl Proof of Proposition~\ref{pro:Orbite-F0-lambda-degre-3-4-5-fermee}]
\textit{(i)} Let us put $\lambda_0=-\frac{1}{d-1}$; according to~(\ref{equa:inflex-F0-lambda-degre-d}) we have $\mathrm{I}_{\F_{0}^{d}(\lambda_0)}^{\hspace{0.2mm}\mathrm{tr}}=(\lambda_0-1)x,$ hence $\deg\mathrm{I}_{\F_{0}^{d}(\lambda_0)}^{\hspace{0.2mm}\mathrm{tr}}=1<d-1$ for any $d\geq3.$ According to the first assertion of Proposition~\ref{pro:cond-nécess-suff-dgnr-F2}, it follows that, for~any~$d\geq3,$ the foliation $\F_{0}^{d}(\lambda_0)$ does not degenerate onto $\F_{2}^{d},$ so that its orbit $\mathcal{O}\big(\F_{0}^{d}(\lambda_0)\big)$ is closed in~$\mathbf{F}(d)$~(Lemma~\ref{lem:Caracterisation-fermeture-Orbite-F0-lambda-degre-d}).
\medskip

\textit{(ii)} Let $[x:y:z]$ be homogeneous coordinates in $\pp.$ For $n\in\N$, let us denote by $\Lambda^{1}_{n}$ the $\C$-vector space of $1$-forms in the variables $x,y,z,$ whose coefficients are homogeneous polynomials of degree $n.$ Let us put $\alpha=y\mathrm{d}z-z\mathrm{d}y,$ $\beta=z\mathrm{d}x-x\mathrm{d}z$ and $\gamma=x\mathrm{d}y-y\mathrm{d}x.$ We have the identification
\begin{small}
\begin{eqnarray*}
\mathbf{F}(d)\hspace{-2.2mm}&=&\hspace{-2.2mm}
\left\{
[\Omega]\in\mathbb{P}(\Lambda^{1}_{d+1})
\hspace{1mm}\vert\hspace{1mm}
\Omega=p\mathrm{d}x+q\mathrm{d}y+r\mathrm{d}z,
\hspace{1mm}
p,q,r\in\C[x,y,z]_{d+1},
\hspace{1mm}
xp+yq+zr=0,\gcd(p,q,r)=1
\right\}
\\
\hspace{-2.2mm}&=&\hspace{-2.2mm}
\left\{
[\Omega]\in\mathbb{P}(\Lambda^{1}_{d+1})
\hspace{1mm}\vert\hspace{1mm}
\Omega=A\alpha+B\beta+C\gamma,
\hspace{1mm}
A,B\in\C[x,y,z]_{d},\hspace{1mm}C\in\C[x,y]_{d},
\hspace{1mm}
\gcd\big(yA-xB,zB-yC,xC-zA\big)=1
\right\}.
\end{eqnarray*}
\end{small}
\hspace{-1mm}By writting
\begin{tiny}
\begin{align*}
&
A=\xi_{1}\hspace{0.1mm}x^d+\xi_3\hspace{0.1mm}x^{d-1}y+\cdots+\xi_{2d+1}\hspace{0.1mm}y^d+
\Big(\xi_{2d+3}\hspace{0.1mm}x^{d-1}+\xi_{2d+5}\hspace{0.1mm}x^{d-2}y+\cdots+\xi_{4d+1}\hspace{0.1mm}y^{d-1}\Big)z+
\Big(\xi_{4d+3}\hspace{0.1mm}x^{d-2}+\xi_{4d+5}\hspace{0.1mm}x^{d-3}y+\cdots+\xi_{6d-1}\hspace{0.1mm}y^{d-2}\Big)z^2+
\cdots+\xi_{d^2+3d+1}\hspace{0.1mm}z^{d},\\
&
B=\xi_{2}\hspace{0.1mm}x^d+\xi_4\hspace{0.1mm}x^{d-1}y+\cdots+\xi_{2d+2}\hspace{0.1mm}y^d+
\Big(\xi_{2d+4}\hspace{0.1mm}x^{d-1}+\xi_{2d+6}\hspace{0.1mm}x^{d-2}y+\cdots+\xi_{4d+2}\hspace{0.1mm}y^{d-1}\Big)z+
\Big(\xi_{4d+4}\hspace{0.1mm}x^{d-2}+\xi_{4d+6}\hspace{0.1mm}x^{d-3}y+\cdots+\xi_{6d}\hspace{0.1mm}y^{d-2}\Big)z^2+
\cdots+\xi_{d^2+3d+2}\hspace{0.1mm}z^{d},\\
&
C=\xi_{d^2+3d+3}\hspace{0.1mm}x^d+\xi_{d^2+3d+4}\hspace{0.1mm}x^{d-1}y+
\xi_{d^2+3d+5}\hspace{0.1mm}x^{d-2}y^2+
\cdots+\xi_{d^2+4d+2}\hspace{0.1mm}xy^{d-1}+\xi_{d^2+4d+3}\hspace{0.1mm}y^d,
\end{align*}
\end{tiny}
\hspace{-1mm}we can identify the class $[\Omega]$ of $\Omega=A\alpha+B\beta+C\gamma$ to the element $[\xi_1:\xi_2:\cdots:\xi_{d^2+4d+3}]\in\mathbb{P}^{\hspace{0.2mm}d^2+4d+2}_{\C}.$ Thus, we can identify $\mathbf{F}(d)$ with the \textsc{Zariski} open set:
\begin{SMALL}
\begin{align*}
\left\{
[\xi_1:\xi_2:\cdots:\xi_{d^2+4d+3}]\in\mathbb{P}^{\hspace{0.2mm}d^2+4d+2}_{\C}
\hspace{0.5mm}\left\vert
\begin{array}[c]{l}
\vspace{2mm}
A=\xi_{1}\hspace{0.1mm}x^d+\xi_3\hspace{0.1mm}x^{d-1}y+\cdots+\xi_{2d+1}\hspace{0.1mm}y^d+
\big(\xi_{2d+3}\hspace{0.1mm}x^{d-1}+\xi_{2d+5}\hspace{0.1mm}x^{d-2}y+\cdots+\xi_{4d+1}\hspace{0.1mm}y^{d-1}\big)z+
\cdots+\xi_{d^2+3d+1}\hspace{0.1mm}z^{d}\\
\vspace{2mm}
B=\xi_{2}\hspace{0.1mm}x^d+\xi_4\hspace{0.1mm}x^{d-1}y+\cdots+\xi_{2d+2}\hspace{0.1mm}y^d+
\big(\xi_{2d+4}\hspace{0.1mm}x^{d-1}+\xi_{2d+6}\hspace{0.1mm}x^{d-2}y+\cdots+\xi_{4d+2}\hspace{0.1mm}y^{d-1}\big)z+
\cdots+\xi_{d^2+3d+2}\hspace{0.1mm}z^{d}\\
\vspace{2mm}
C=\xi_{d^2+3d+3}\hspace{0.1mm}x^d+\xi_{d^2+3d+4}\hspace{0.1mm}x^{d-1}y+\xi_{d^2+3d+5}\hspace{0.1mm}x^{d-2}y^2+
\cdots+\xi_{d^2+4d+2}\hspace{0.1mm}xy^{d-1}+\xi_{d^2+4d+3}\hspace{0.1mm}y^d\\
\gcd\big(yA-xB,zB-yC,xC-zA\big)=1
\end{array}
\right.
\right\}.
\end{align*}
\end{SMALL}
\hspace{-1.17mm}Then, via this identification, we have
\begin{small}
\begin{align*}
&
\F_{2}^{d}=\big[\Omega_{2}^{d}\big]=\big[x^{d}\beta+y^{d}\gamma]=[0:1:0:0:\cdots:0:0:1\big]
\\
&\hspace{-0.7cm}\text{\normalsize and}
\\
&
\F_{0}^{d}(\lambda)=\big[\Omega_{0}^{d}(\lambda)\big]=\big[(y^d+xz^{d-1})\alpha+\lambda\,yz^{d-1}\beta\big]                              =\big[\underbrace{0:0:\cdots:0}_{2d}:1:\underbrace{0:0:\cdots:0}_{d^2+d-5}:1:0:0:\lambda:\underbrace{0:0:\cdots:0}_{d+3}\big].
\end{align*}
\end{small}
\hspace{-1.17mm}In addition, the orbit of a foliation $\F=[\Omega]\in\mathbf{F}(d)$ is
\begin{align*}
\mathcal{O}(\F)=\left\{
[\varphi^*\Omega]
\hspace{1mm}\Big\vert\hspace{1mm}
\varphi=[a_1x+a_2y+a_3z:a_4x+a_5y+a_6z:a_7x+a_8y+a_9z]\in\mathrm{Aut}(\pp)
\right\}.
\end{align*}

\noindent Let $[x_1:x_2:\cdots:x_{d^2+4d+3}]$ be a system of homogeneous coordinates in $\mathbb{P}^{\hspace{0.2mm}d^2+4d+2}_{\C}.$ For $d=3,$  let us consider the following homogeneous polynomial in $x_1,x_2,\ldots,x_{24}$ of degree $5$:
\begin{SMALL}
\begin{align*}
&\hspace{0.3cm}
P_{3}=
-90x_{2}\Big(x_{1}\left(294x_{1}-269x_{4}\right)+10x_{2}\left(29x_{3}+4x_{6}\right)+86x_{4}^2\Big)x_{22}x_{24}
-1125x_{2}^2\left(21x_{1}-23x_{4}\right)x_{23}x_{24}\\
&\hspace{0.87cm}
+45x_{2}\Big(2x_{3}\left(294x_{1}+13x_{4}\right)-x_{6}\left(552x_{1}-271x_{4}\right)+1125x_{2}x_{5}\Big)x_{21}x_{24}+28125x_{2}x_{10}x_{21}x_{23}x_{24}\\
&\hspace{0.87cm}
+25\Big(108\left(x_{9}-2x_{12}\right)\left(3x_{1}-4x_{4}\right)+9x_{10}\left(112x_{3}-93x_{6}\right)
+675x_{2}x_{11}\Big)x_{21}^2x_{24}-6000x_{2}x_{10}x_{22}^2x_{24}\\
&\hspace{0.87cm}
-5625x_{5}x_{11}x_{21}^3+20\Big(\left(2x_{1}-x_{4}\right)\left(41x_{9}-7x_{12}\right)+30x_{10}\left(2x_{3}-3x_{6}\right)+50x_{2}x_{11}\Big)x_{22}^3
-50625x_{2}^3x_{24}^2\\
&\hspace{0.87cm}
-5\Big(2x_{9}\left(207x_{1}-116x_{4}\right)-x_{12}\left(153x_{1}-314x_{4}\right)+5x_{10}\left(356x_{3}-359x_{6}\right)+1350x_{2}x_{11}\Big)x_{21}x_{22}x_{23}\\
&\hspace{0.87cm}
+1875\Big(x_{11}\left(2x_{3}-x_{6}\right)+x_{5}\left(2x_{9}-x_{12}\right)\Big)x_{21}^2x_{22}
-375x_{2}\Big(2x_{1}\left(3x_{1}-7x_{4}\right)-x_{2}\left(3x_{3}-2x_{6}\right)+8x_{4}^2\Big)x_{23}^2\\
&\hspace{0.87cm}
+50\Big(5x_{10}\left(39x_{1}-38x_{4}\right)-3x_{2}\left(x_{9}-32x_{12}\right)\Big)x_{21}x_{23}^2
-50\Big(x_{10}\left(14x_{1}-37x_{4}\right)-3x_{2}\left(7x_{9}+x_{12}\right)\Big)x_{22}^2x_{23}\\
&\hspace{0.87cm}
+15\Big(5x_{11}\left(21x_{1}+22x_{4}\right)-8x_{3}\left(14x_{9}-43x_{12}\right)+6x_{6}\left(13x_{9}-56x_{12}\right)-350x_{5}x_{10}\Big)x_{21}^2x_{23}+R\,x_{21}^2\\
&\hspace{0.87cm}
-5\Big(20x_{11}\left(24x_{1}-7x_{4}\right)+4x_{9}\left(97x_{3}-43x_{6}\right)+x_{12}\left(94x_{3}-211x_{6}\right)-600x_{5}x_{10}\Big)x_{21}x_{22}^2+S\,x_{21}x_{22}\\
&\hspace{0.87cm}
-75\Big(2x_{10}\left(78x_{1}-29x_{4}\right)-15x_{2}\left(2x_{9}-19x_{12}\right)\Big)x_{21}x_{22}x_{24}+125x_{2}x_{10}x_{22}x_{23}^2+Tx_{22}^2+Ux_{21}x_{23}\\
&\hspace{0.87cm}
+Vx_{22}x_{23},
\end{align*}
\end{SMALL}
\hspace{-1mm}where
\begin{SMALL}
\begin{align*}
&\hspace{0.3cm}
R=5568x_{6}x_{5}\left(3x_{1}-4x_{4}\right)-18x_{3}x_{5}\left(1612x_{1}-1941x_{4}\right)+6x_{3}^2\left(1952x_{3}-4389x_{6}\right)+3x_{6}^2\left(7057x_{3}-2136x_{6}\right)
-11250x_{2}x_{5}^2\\
&\hspace{0.82cm}
+2700x_{7}\left(3x_{1}-4x_{4}\right)^2+54x_{8}\left(3x_{1}-4x_{4}\right)\left(106x_{3}-89x_{6}\right),
\\
&\hspace{0.3cm}
S=27000x_{2}x_{7}\left(3x_{1}-4x_{4}\right)-24x_{3}^2\left(658x_{1}-249x_{4}\right)+1512x_{4}x_{8}\left(11x_{1}-4x_{4}\right)+252x_{1}^2\left(83x_{5}-36x_{8}\right)
-90x_{2}x_{3}\left(329x_{5}-318x_{8}\right)\\
&\hspace{0.82cm}
-2x_{4}x_{5}\left(17073x_{1}-6047x_{4}\right)+3x_{1}x_{6}\left(8712x_{3}-3599x_{6}\right)-x_{4}x_{6}\left(11658x_{3}-6041x_{6}\right)
+90x_{2}x_{6}\left(226x_{5}-267x_{8}\right),\\
&\hspace{0.3cm}
T=20x_{1}x_{3}\left(294x_{1}-253x_{4}\right)-40x_{1}x_{6}\left(159x_{1}-152x_{4}\right)+1900x_{2}x_{3}\left(x_{3}-x_{6}\right)
+20x_{4}^2\left(68x_{3}-95x_{6}\right)-25x_{2}x_{6}\left(40x_{3}-33x_{6}\right)\\
&\hspace{0.82cm}
+60x_{1}x_{2}\left(361x_{5}-252x_{8}\right)-10x_{2}x_{4}\left(983x_{5}-756x_{8}\right)+67500x_{2}^2x_{7},\\
&\hspace{0.3cm}
U=90x_{1}x_{3}\left(98x_{1}-117x_{4}\right)-30x_{1}x_{6}\left(171x_{1}-284x_{4}\right)-150x_{2}x_{6}\left(68x_{3}-35x_{6}\right)
-30x_{2}x_{4}\left(167x_{5}+396x_{8}\right)+7050x_{2}x_{3}^2\\
&\hspace{0.82cm}
+20x_{4}^2\left(73x_{3}-157x_{6}\right)+270x_{1}x_{2}\left(41x_{5}+33x_{8}\right),\\
&\hspace{0.3cm}
V=5x_{2}x_{4}\left(1604x_{3}-611x_{6}\right)-30x_{1}^2\left(294x_{1}-563x_{4}\right)-30x_{4}^2\left(355x_{1}-86x_{4}\right)
-30x_{1}x_{2}\left(463x_{3}-242x_{6}\right)-75x_{2}^2\left(109x_{5}-198x_{8}\right).
\end{align*}
\end{SMALL}
\hspace{-1mm}A computation carried out with \texttt{Maple} shows that evaluating $P_{3}$ at an arbitrary element $[\xi_1:\xi_2:\cdots:\xi_{24}]$ of $\mathcal{O}\big(\F_{0}^{3}(\lambda)\big),$ we find $P_{3}\big([\xi_1:\xi_2:\cdots:\xi_{24}]\big)=0,$ {\it i.e.} $\mathcal{O}\big(\F_{0}^{3}(\lambda)\big)$ is contained in the zero locus of $P_{3}$
\begin{align*}
\mathrm{\text{Zeros}}(P_{3}):=
\left\{[x_1:x_2:\cdots:x_{24}]\in\mathbb{P}^{23}_{\C}\hspace{1mm}\vert\hspace{1mm}P_{3}\big([x_1:x_2:\cdots:x_{24}]\big)=0\right\},
\end{align*}
which is a \textsc{Zariski} closed subset of $\mathbb{P}^{23}_{\C}.$ Therefore we have $\overline{\mathcal{O}\big(\F_{0}^{3}(\lambda)\big)}\subset\mathrm{\text{Zeros}}(P_{3})$ for any $\lambda\in\C^*.$ Moreover, we have
\begin{align*}
P_{3}\left(0,1,0,0,\cdots,0,0,1\right)=-50625\neq0,
\end{align*}
hence $\F_{2}^{3}\not\in\mathrm{\text{Zeros}}(P_{3}).$ It follows that, for any $\lambda\in\C^*,$ we have $\F_{2}^{3}\not\in\overline{\mathcal{O}\big(\F_{0}^{3}(\lambda)\big)}$, so that $\F_{0}^{3}(\lambda)$ does not degenerate onto $\F_{2}^{3}.$ Consequently, according to Lemma~\ref{lem:Caracterisation-fermeture-Orbite-F0-lambda-degre-d}, the orbit $\mathcal{O}\big(\F_{0}^{3}(\lambda)\big)$ is closed in~$\mathbf{F}(3)$.

\noindent To show that the orbit $\mathcal{O}\big(\F_{0}^{4}(\lambda)\big),$ resp. $\mathcal{O}\big(\F_{0}^{5}(\lambda)\big),$ is closed in $\mathbf{F}(4),$ resp. $\mathbf{F}(5),$ it suffices to argue as in degree $d=3$, replacing the polynomial $P_3$ by the following polynomial $P_4,$ resp. $P_5$:
\begin{SMALL}
\begin{align*}
&
P_{4}=
\Big(3x_{3}\left(129x_{3}-212x_{6}\right)+3x_{4}\left(178x_{5}+15x_{8}\right)+12x_{1}\left(22x_{5}-3x_{8}\right)+5184x_{2}x_{7}-20x_{6}^2\Big)x_{31}
+1728x_{15}x_{31}^2\\
&\hspace{0.57cm}
-432\left(2x_{13}-x_{16}\right)x_{31}x_{32}+48\left(42x_{11}-31x_{14}\right)x_{31}x_{33}-18\left(24x_{11}-19x_{14}\right)x_{32}^2
-162x_{2}\left(4x_{1}-15x_{4}\right)x_{34}\\
&\hspace{0.57cm}
-18\Big(2x_{1}\left(27x_{3}-20x_{6}\right)-x_{4}\left(15x_{3}-x_{6}\right)+x_{2}\left(170x_{5}-69x_{8}\right)\Big)x_{32}
+4212x_{12}x_{31}x_{34}-486x_{12}x_{32}x_{33}\\
&\hspace{0.57cm}
+36\Big(3\left(x_{1}-x_{4}\right)\left(12x_{1}-x_{4}\right)+22x_{2}\left(3x_{3}-2x_{6}\right)\Big)x_{33}-10368x_{2}^2x_{35},
\\
\\
\text{\normalsize resp.}\hspace{1mm}&
P_{5}=
\Big(50x_{7}\left(4906x_{1}-4749x_{4}\right)-27040x_{10}\left(5x_{1}-6x_{4}\right)-5x_{5}\left(10596x_{3}-13469x_{6}\right)+20x_{8}\left(1019x_{3}-2028x_{6}\right)\\
&\hspace{0.57cm}
+569100x_{2}x_{9}\Big)x_{43}+142275x_{19}x_{43}^2-11690x_{17}x_{43}x_{44}+98140x_{14}x_{43}x_{47}-140x_{2}\left(2180x_{1}-1691x_{4}\right)x_{47}\\
&\hspace{0.57cm}
+35\left(1564x_{13}-1645x_{16}\right)x_{43}x_{46}+\Big(8620x_{8}\left(2x_{1}-x_{4}\right)-50x_{5}\left(141x_{1}-11x_{4}\right)+10x_{3}\left(513x_{3}-1580x_{6}\right)\\
&\hspace{0.57cm}
+70x_{2}\left(2779x_{7}-2704x_{10}\right)+9875x_{6}^2\Big)x_{44}
-35\Big(\left(x_{1}-x_{4}\right)\left(295x_{1}+683x_{4}\right)-x_{2}\left(3776x_{3}-4427x_{6}\right)\Big)x_{46}\\
&\hspace{0.57cm}
+70\left(323x_{18}-253x_{15}\right)x_{43}x_{45}+7\left(686x_{13}-293x_{16}\right)x_{44}x_{45}-2975x_{15}x_{44}^2-15946x_{14}x_{45}^2-1422750x_{2}^2x_{48}\\
&\hspace{0.57cm}
+\Big(14x_{3}\left(15x_{1}+1124x_{4}\right)-14x_{6}\left(10x_{1}+1129x_{4}\right)-595x_{2}\left(221x_{5}-250x_{8}\right)\Big)x_{45}+49210x_{14}x_{44}x_{46}.
\end{align*}
\end{SMALL}
\end{proof}

\vspace{2.2cm}

\noindent For $d\geq6,$ we propose:
\begin{conj}\label{conj:1}
{\sl
Let $d$ be an integer greater than or equal to $6$ and $\lambda$ a nonzero complex number. A homogeneous coordinate system $[x_1:x_2:\cdots:x_{d^2+4d+3}]$ being fixed in $\mathbb{P}^{\hspace{0.2mm}d^2+4d+2}_{\C},$ there exists a homogeneous polynomial $Q_d\in\C[x_1,x_2,\cdots,x_{d^2+4d+3}]$ of degree $3$, not depending on $\lambda,$ which vanishes on the orbit $\mathcal{O}\big(\F_{0}^{d}(\lambda)\big)$ and does not vanish at the point~$\F_{2}^{d}=[0:1:0:0:\cdots:0:0:1\big].$
}
\end{conj}

\noindent Computations made with \texttt{Maple} by the first author show the validity of this conjecture for $ d $ small ($ d \leq30 $) by taking the polynomial $ Q_d $ in the following form:
\begin{Small}
\begin{align*}
&
Q_d=x_{d^2+3d+3}\left(\sum\limits_{i=1}^{d-1}\alpha_{i}\,x_{2d+2i+1}\,x_{d^2+4d+2-i}+\sum\limits_{i=0}^{4}\beta_{i}\,x_{2d+2i+4}\,x_{d^2+4d+2-i}\right)
+(x_1\hspace{2mm} x_2\hspace{2mm} \cdots\hspace{2mm}x_{d+1})M\left(\begin{array}{c}x_{d^2+4d+3}\\x_{d^2+4d+2}\\\vdots \\x_{d^2+3d+3}\end{array}\right)\\
&\hspace{0.72cm}
+x_{d^2+3d+4}\left(\delta_{0}\,x_{2d+4}\,x_{d^2+4d+1}+\delta_{1}\,x_{2d+6}\,x_{d^2+4d}+\sum\limits_{i=1}^{d-3}\gamma_{i}\,x_{2d+2i+1}\,x_{d^2+4d+1-i}\right),
\end{align*}
\end{Small}
\hspace{-1mm}where $M=\left(\begin{array}{c}L_{1}\\L_{2}\\\vdots\\L_{d+1}\end{array}\right)$ is a square matrix of order $ d + 1 $ whose lines are of the form:
\begin{SMALL}
\begin{align*}
&
L_{1}=\left[
0
\hspace{2mm}
0
\hspace{2mm}
a_{1,3}x_{1}+b_{1,3}x_{4}
\hspace{2mm}
a_{1,4}x_{3}+b_{1,4}x_{6}
\hspace{2mm}
a_{1,5}x_{5}+b_{1,5}x_{8}
\hspace{2mm}
\cdots
\hspace{2mm}
a_{1,\,d+1}x_{2d-3}+b_{1,\,d+1}x_{2d}
\right]
\\
&
L_{2}=\left[
b_{2,1}x_{2}
\hspace{2mm}
a_{2,2}x_{1}+b_{2,2}x_{4}
\hspace{2mm}
a_{2,3}x_{3}+b_{2,3}x_{6}
\hspace{2mm}
a_{2,4}x_{5}+b_{2,4}x_{8}
\hspace{2mm}
a_{2,5}x_{7}+b_{2,5}x_{10}
\hspace{2mm}
\cdots
\hspace{2mm}
a_{2,\,d+1}x_{2d-1}+b_{2,\,d+1}x_{2d+2}
\right]
\\
&
\vdots
\\
&
L_{2k-1}=\left[
\underbrace{0\hspace{2mm}0\hspace{2mm}\cdots\hspace{2mm}0}_{\min(2k,\,d+1)}
\hspace{2mm}
a_{2k-1,\,2k+1}x_{2k-1}+b_{2k-1,\,2k+1}x_{2k+2}
\hspace{2mm}
a_{2k-1,\,2k+2}x_{2k+1}+b_{2k-1,\,2k+2}x_{2k+4}
\hspace{2mm}
\cdots
\hspace{2mm}
a_{2k-1,\,d+1}x_{2d-2k-1}+b_{2k-1,\,d+1}x_{2d-2k+2}
\right]
\\
&
L_{2k}=\left[
\underbrace{0\hspace{2mm}0\hspace{2mm}\cdots\hspace{2mm}0}_{2k-2}
\hspace{2mm}
b_{2k,\,2k-1}x_{2k}
\hspace{2mm}
a_{2k,\,2k}x_{2k-1}+b_{2k,\,2k}x_{2k+2}
\hspace{2mm}
a_{2k,\,2k+1}x_{2k+1}+b_{2k,\,2k+1}x_{2k+4}
\hspace{2mm}
\cdots
\hspace{2mm}
a_{2k,\,d+1}x_{2d-2k+1}+b_{2k,\,d+1}x_{2d-2k+4}
\right],
\end{align*}
\end{SMALL}
\hspace{-1mm}where $\alpha_i,\beta_i,\gamma_i,\delta_i,a_{i,j},b_{i,j}\in\C$ with $b_{2,1}\neq0.$
\vspace{3mm}

\noindent It is clear that Conjecture~\ref{conj:1} and Lemma~\ref{lem:Caracterisation-fermeture-Orbite-F0-lambda-degre-d} imply the following conjecture.
\begin{conj}\label{conj:2}
{\sl
For any integer $d\geq6$ and any $\lambda\in\C^*,$ the orbit $\mathcal{O}\big(\F_{0}^{d}(\lambda)\big)$ is closed in $\mathbf{F}(d).$
}
\end{conj}


\end{document}